\documentclass[11pt]{amsart}

\usepackage{amsmath,amsfonts,amssymb}
\usepackage{fullpage}

\usepackage{ifxetex}
\ifxetex
  \usepackage{fontspec}
\else
  \usepackage[utf8]{inputenc}
\fi

\renewcommand{\geq}{\geqslant}
\renewcommand{\leq}{\leqslant}
\renewcommand{\preceq}{\preccurlyeq}
\renewcommand{\succeq}{\succcurlyeq}

\usepackage{paralist}
\usepackage{tropical}
\usepackage{tikz}
\usepackage{hyperref}

\usetikzlibrary{arrows}
\usetikzlibrary{shapes}
\usetikzlibrary{matrix}
\usetikzlibrary{decorations.markings}
\usetikzlibrary{decorations.pathreplacing}
\usetikzlibrary{fadings}
\usetikzlibrary{automata,calc}
\usetikzlibrary{patterns}

\usepackage{algorithmic}
\renewcommand{\algorithmicrequire}{\textbf{Input:}}
\renewcommand{\algorithmicensure}{\textbf{Output:}}

\newcommand*\tto[1][]{\overset{#1}{\longrightarrow}}

\newcommand*\Int{\mathbb{Z}}
\newcommand*\Nat{\mathbb{N}}

\newcommand*\Real{\mathbb{R}}
\newcommand*\Realnn{\Real_{ \ge 0}}

\newcommand{\etc}{\textit{etc}}
\newcommand{\R}{\mathbb{R}}
\newcommand{\mydef}{:=}
\renewcommand{\defi}{\mydef}
\newcommand{\smaxplus}{\mathbb{G}}
\newcommand{\sleq}{\mathrel{\leq_{\smaxplus}}}
\newcommand{\sgeq}{\mathrel{\geq_{\smaxplus}}}
\newcommand{\RRleq}{\mathrel{\leq_{\RR}}}
\newcommand{\RRgeq}{\mathrel{\geq_{\RR}}}
\newcommand{\RRmpzero}{\mathbbb{0}}
\newcommand{\RRmpone}{\mathbbb{1}}
\newcommand{\sgreater}{\mathrel{>_{\smaxplus}}}
\newcommand{\sless}{\mathrel{<_{\smaxplus}}}

\newcommand{\RR}{\mathbb{S}}
\newcommand{\mcR}{\mathcal{R}}
\newcommand{\mcS}{\mathcal{S}}

\newcommand{\eps}{\varepsilon}

\newcommand{\myul}[1]{#1^{-}}

\DeclareMathOperator{\tconv}{\text{\normalfont tconv}}
\DeclareMathOperator{\tcone}{\text{\normalfont tcone}}
\newcommand{\GG}{\mathcal{G}}

\newcommand{\mpglib}{\text{MPGLib}}
\renewcommand{\mpzero}{-\infty}
\renewcommand{\mpone}{0}

\newcommand{\binomial}[2]{\binom{#1}{#2}}
\newcommand{\pert}{\R^{-}}
\newcommand{\abs}[1]{\lvert#1\rvert}

\newtheorem{proposition}{Proposition}

\newtheorem{lemma}[proposition]{Lemma}
\newtheorem{theorem}[proposition]{Theorem}
\newtheorem{corollary}[proposition]{Corollary}
\newtheorem{assumption}{Assumption}

\theoremstyle{remark}
\newtheorem{remark}[proposition]{Remark}
\newtheorem{example}[proposition]{Example}

\title{Tropical Fourier-Motzkin Elimination, with an Application to
  Real-Time Verification}
\thanks{The five authors were partially supported by the programme ``Ingénierie Numérique \& Sécurité'' of the French National Agency of
  Research (ANR), project ``MALTHY'', number ANR-13-INSE-0003. Xavier Allamigeon was partially supported by the same programme, project ``CAFEIN'', number ANR-12-INSE-0007. Ricardo
  D.~Katz was partially supported by CONICET Grant PIP
  112-201101-01026. Xavier Allamigeon, Ricardo D.~Katz and Stéphane
  Gaubert were partially supported by the PGMO program of EDF and
  Fondation Math\'ematique Jacques Hadamard.  Uli Fahrenberg and Axel
  Legay were partially supported by the EU IP project DANSE and by the
  regional CREATE Estase initiative.}

\author{Xavier Allamigeon}
\email[X.~Allamigeon]{xavier.allamigeon@inria.fr}
\author{Uli Fahrenberg}
\email[U.~Fahrenberg]{ulrich.fahrenberg@inria.fr}
\author{Stéphane Gaubert}
\email[S.~Gaubert]{stephane.gaubert@inria.fr} 
\author{Ricardo D.~Katz}
\email[R.~D.~Katz]{katz@cifasis-conicet.gov.ar}
\author{Axel Legay}
\email[A.~Legay]{axel.legay@inria.fr}

\address[X.~Allamigeon and S.~Gaubert]{INRIA and CMAP, \'Ecole Polytechnique, CNRS, 91128 Palaiseau C\'edex, France}
\address[U.~Fahrenberg and A.~Legay]{Irisa / INRIA, Campus
  Beaulieu, 35042 Rennes C\'edex, France}
\address[R.~D.~Katz]{CONICET-CIFASIS, Bv. 27 de Febrero 210 bis, 2000 Rosario, Argentina}

\date{\today}
\keywords{Tropical polyhedra, strict inequalities, 
    Fourier-Motzkin elimination, mean payoff games, real-time
    verification, timed automata}
    
\catcode`\@=11
\@namedef{subjclassname@2010}{%
\textup{2010} Mathematics Subject Classification}
\catcode`\@=12
\subjclass[2010]{14T05, 52A01, 52B55}

\begin{document}

\begin{abstract}
We introduce a generalization of tropical polyhedra able to express both strict and non-strict inequalities. Such inequalities are handled by means of a semiring of germs (encoding infinitesimal perturbations). 
We develop a tropical analogue of Fourier-Motzkin elimination
from which we derive geometrical properties of these polyhedra.
In particular, we show that they coincide with the tropically convex union of (non-necessarily closed) cells
that are convex both classically and tropically. We also prove that the redundant inequalities produced when performing
successive elimination steps
can be dynamically deleted by reduction to mean payoff game problems. 
As a complement, we provide 
a coarser (polynomial time) deletion procedure  
which is enough
to arrive at a simply exponential bound for the total execution time.
These algorithms are illustrated
by an application to real-time systems
(reachability analysis of timed automata).
\end{abstract}

\maketitle

\section{Introduction}\label{sec:intro}

\subsection*{Tropical convexity} 
\emph{Tropical} or \emph{max-plus algebra} refers to the set $\maxplus
\defi \Real \cup \{-\infty\}$ equipped with $x \mpplus y \defi
\max(x,y)$ as addition and $x \mptimes y \defi x+y$ as multiplication
(the latter will be also denoted by concatenation $x y$).  In this setting,
an inequality constraint on variables $\vect x_1,\dots, \vect x_n$ is
said to be \emph{(tropically) affine} if it is of the form:
\begin{equation}\label{eq:affine_inequality}
a_0 \mpplus a_1 \vect{x}_1 \mpplus \cdots \mpplus a_n \vect{x}_n \leq b_0 \mpplus b_1 \vect{x}_1 \mpplus \cdots \mpplus b_n \vect{x}_n \enspace, 
\end{equation}
or equivalently, with usual notation,
\begin{equation}\label{eq:affine_inequality2}
\max(a_0, a_1 + \vect{x}_1, \dots, a_n + \vect{x}_n) \leq \max(b_0, b_1 + \vect{x}_1, \dots, b_n + \vect{x}_n) \enspace,
\end{equation}
where $a_i, b_i \in \maxplus$ for $i =0,1, \ldots , n$. By analogy with the terminology of usual convex geometry, a \emph{tropical (convex) polyhedron} is defined as a set composed of all the vectors $\vect{x} \in \maxplus^n$ satisfying finitely many such inequalities. An example is depicted on the left-hand side of Figure~\ref{fig:zone}.

\begin{figure}[t]
\begin{center}
\begin{tikzpicture}[>=stealth',initial text=,every state/.style={minimum size=0.4cm},convex/.style={draw=lightgray,fill=lightgray,fill opacity=0.7},convexborder/.style={very thick},scale=1]
\begin{scope}[scale=.6]
\draw[gray!40,very thin] (-3.5,-2.5) grid (4.5,4.5);
\draw[gray!80,->] (-3.5,0) -- (4.5,0) node[color=gray!80,above] {$\vect{x}_1$};
\draw[gray!80,->] (0,-2.5) -- (0,4.5) node[color=gray!80,right] {$\vect{x}_2$};
\filldraw[convex] (-1,0) -- (-1,-2) -- (1,-2) -- (2,-1) -- (2,1) -- (4.5, 3.5) -- (4.5,4.5) -- (2.5,4.5) -- (0,2) -- (-3,2) -- (-3,0) -- cycle;
\draw[convexborder] (-1,0) -- (-1,-2) -- (1,-2) -- (2,-1) -- (2,1) -- (4.5, 3.5)  (2.5,4.5) -- (0,2) -- (-3,2) -- (-3,0) -- (-1,0) ;  
\end{scope}
\begin{scope}[scale=.6,xshift=10cm,yshift=-1cm]
\coordinate (v0) at (0,0);
\coordinate (v1) at (0,2);
\coordinate (v2) at (2,4);
\coordinate (v3) at (6,4);
\coordinate (v4) at (6,3);
\coordinate (v5) at (3,0);

\draw[gray!40,very thin] (-1.5,-1.5) grid (7.5,5.5);
\draw[gray!80,->] (-1.5,-1) -- (7.5,-1) node[color=gray!80,above] {$\vect{x}_1$};
\draw[gray!80,->] (-1,-1.5) -- (-1,5.5) node[color=gray!80,right] {$\vect{x}_2$};

\fill[convex] (v0) -- (v1) -- (v2) -- (v3) -- (v4) -- (v5) -- cycle;
\draw[convexborder] (v0) -- (v1) (v2) -- (v3)  (v4) -- (v5);
\end{scope}
\end{tikzpicture}
\end{center}
\caption{Left: a tropical polyhedron (including the black border). Right: a (non-closed) zone defined by the inequalities $1 \leq \vect{x}_1 < 7$, $1 < \vect{x}_2 \leq 5$, $-2 < \vect{x}_1 - \vect{x}_2 \leq 3$.} \label{fig:zone}
\end{figure}
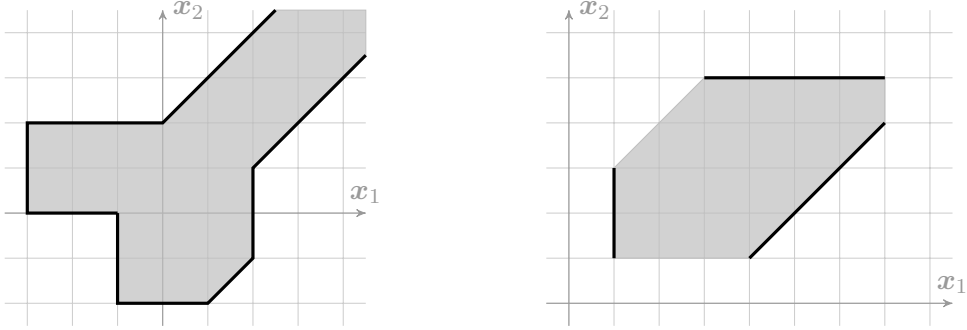

Tropical polyhedra and, more generally, tropically convex sets, have
been introduced and studied in various contexts, 
including optimization~\cite{zimmermann77}, control
theory~\cite{ccggq99},
idempotent functional analysis~\cite{litvinov00}, or combinatorics~\cite{DS}. Several basic results of convex analysis and geometry have been 
shown to have tropical analogues. These include
Hahn-Banach~\cite{zimmermann77,litvinov00,cgqs05,DS}, Minkowski~\cite{GK,BSS,joswig04}, and Carath{\'e}odory/Helly-type~\cite{BriecHorvath04,gauser,GM08} theorems. Some algorithmic aspects have also been studied (\eg~\cite{butkovicH,Joswig2008,AllamigeonGaubertGoubaultDCG2013}). We refer the reader to~\cite{AllamigeonGaubertGoubaultDCG2013} for further references.

\subsection*{Motivation} The present work is motivated by a specific application of tropical algebra to the verification of real-time systems. Indeed, 
a remarkable property of tropical polyhedra is their ability to
concisely encode possibly non-convex sets expressed as disjunctions of
closed zones. A \emph{closed zone}, also known in the literature as \emph{polytrope},
is a set of vectors $\vect{x} \in \R^n$ defined by inequalities of
the form $\vect x_i\geq m_i$, $\vect x_i\leq M_i$, and $ \vect{x}_i \leq k_{ij} + \vect{x}_j$, for certain constants $m_i, M_i, k_{ij} \in \R$. 
More generally, zones are obtained by replacing some of the previous
inequalities by strict ones. See the right-hand side of
Figure~\ref{fig:zone} for an illustration.

Zones are extensively used in the area of verification of 
real-time systems, where these systems are modelled by formalisms such as \eg~timed
automata~\cite{journal/tcs/AlurDd94}
or timed Petri
nets~\cite{conf/ajwsm/Bowden96}.
More precisely, zones are used by model checking tools as \emph{symbolic states},
typically representing infinitely many 
states of the system. 
They can be represented using difference-bound matrices (DBM), which
are essentially adjacency matrices of weighted graphs. This allows for
efficient algorithms for the manipulation of zones during the
verification process.  

An inherent drawback of zones is that they are convex sets, and
consequently they are not closed under set union. This means that during
the analysis process, symbolic states cannot generally be combined,
which potentially leads to state-space explosion. Due to this, tropical
polyhedra have been proposed in~\cite{journals/jlap/LuMMRFL} as a
replacement for zones. However, an important drawback in this approach
is that the analysis of timed automata often requires to express strict
constraints, for instance in the analysis of communication
protocols~\cite{inbook/DavidILS10,DBLP:journals/isse/Kot09}, while
tropical polyhedra are by definition topologically closed.  An example
illustrating these drawbacks will be given in
Section~\ref{sec:tropical_reachability_analysis}.

\subsection*{Contributions} In this paper we first introduce (Section~\ref{sec:mixed_constraints}) a class of non-necessarily closed tropically convex sets. This class is called \emph{tropical polyhedra with mixed constraints}. It can express not only inequalities of the form~\eqref{eq:affine_inequality2} in which the relation $\leq$ has been replaced by $<$, but also finer constraints exploiting the disjunctive character of tropical inequalities. For instance, the inequality
\[
\vect{x}_1 \leq \max(1+\vect{x}_2, \myul{0} + \vect{x}_3)
\]
is going to represent the disjunction of the inequalities $\vect{x}_1 \leq 1+\vect{x}_2$ and $\vect{x}_1 < \vect{x}_3$.
These \emph{mixed inequalities} are defined using coefficients in a semiring of affine germs, which  represent infinitesimal perturbations of reals. 
 
In the second place, we present a tropical counterpart of
Fourier-Motzkin elimination (Section~\ref{sec:fourier_motzkin}). It provides a constructive method to show that the projection in $\maxplus^{n-1}$ of a tropical polyhedron with mixed constraints $\PP \subset \maxplus^n$  is a polyhedron with mixed constraints  (Theorem~\ref{th:fourier_motzkin_maxplus_case}). It computes a representation by mixed inequalities of the projection by combining the defining inequalities of $\PP$. 
Actually, this approach handles more generally
systems of inequalities with coefficients in a totally ordered 
idempotent semiring, modulo some assumptions.
Note that such an analogue of Fourier-Motzkin
algorithm has not been considered previously in the tropical setting,
even in the case of standard (closed) tropical polyhedra.
Fourier-Motzkin elimination also appears as a useful tool to show the polyhedral character of some non-closed tropically convex sets. As an application, we indeed prove that tropical polyhedra with mixed constraints are precisely the tropically convex union of finitely many 
zones, and the intersection of finitely many tropical hemispaces (\ie{}~tropically convex sets whose complements are also tropically convex, which were studied in \eg~\cite{BH-08,KNS}), see Theorem~\ref{th:union_of_zones} and Corollary~\ref{cor:intersection_of_hemispaces}.

Superfluous inequalities may be produced by Fourier-Motzkin algorithm, so
that the size of the constraint systems can grow in a double exponential
way during consecutive applications of the method. In order to eliminate such redundant inequalities, 
in Section~\ref{subsec:elimination_step} we extend to mixed 
inequalities a result of~\cite{AkianGaubertGutermanIJAC2011}
and its subsequent refinement in~\cite{AllamigeonGaubertKatzLAA2011}, 
building
on techniques of these two papers.
The result of~\cite{AkianGaubertGutermanIJAC2011}
shows that deciding the feasibility of a system
of tropical linear inequalities
is (Karp) polynomial-time equivalent to solving mean payoff games.
The result of~\cite{AllamigeonGaubertKatzLAA2011}
shows that deciding logical implications
over tropical linear inequalities is also 
equivalent to solving mean payoff games. Theorem~\ref{th:equivalence_MPG} generalizes these two results to mixed inequalities. 
We note that the present approach (through germs) also yields an
alternative, simpler derivation of the result of~\cite{AllamigeonGaubertKatzLAA2011}. 
Indeed, deciding whether a given inequality of the form~\eqref{eq:affine_inequality2} is logically implied by a system of other inequalities of the same kind amounts to checking if the intersection of a tropical polyhedron with the complement of a closed half-space is empty or not. Such an intersection is obviously a tropical polyhedron with mixed constraints.

Experimentally efficient algorithms have been developed to solve mean payoff games, 
but no polynomial time algorithm is known.
Hence we also provide a weak criterion which allows to eliminate some of
the superfluous inequalities in polynomial time
(Section~\ref{subsec:weak_criterion}). We prove that, in the case of
non-strict inequalities, this weak elimination is sufficient to obtain a
\emph{single}-exponential bound
for Fourier-Motzkin elimination
(Section~\ref{subsec:complexity}).

Finally, Section~\ref{sec:tropical_reachability_analysis} illustrates the
application of tropical polyhedra with mixed constraints to the verification (reachability analysis) of timed automata. 
We show that the operations necessary for forward
exploration of timed automata can be defined on tropical polyhedra with
mixed constraints, using Fourier-Motzkin elimination and the algorithms
developed to eliminate redundant inequalities.

\subsection*{Related work}

The algorithms developed so far for tropical polyhedra usually benefit
from the fact that these can be represented either \emph{externally},
using inequalities (as in~\eqref{eq:affine_inequality}), or \emph{internally}, as sets generated by
finitely many points and rays
(see~\cite{GK09} for details). In contrast, non-closed tropically convex sets may not be finitely
generated. Generating representations of (non-necessarily closed) tropical convex cones have been studied in~\cite{BSS}, and in~\cite{KNS} in the particular case of tropical hemispaces. 
A certain class of possibly infinite generating
representations was treated in~\cite{GaubertKatzKybernetica2004},
however, the associated algorithms rely on the expensive Presburger
arithmetic. Defining non-closed polyhedra using infinitesimal
perturbations of generators also presents some difficulties, see
Remark~\ref{remark:generators} below. Moreover, we should warn the reader
that some
geometric aspects of tropical polyhedra, in particular the notion of
faces, are still not yet understood~\cite{DevelinYu}. 
Thus it does not
seem easy to manipulate non-closed polyhedra from closed ones by
excluding some ``facets'' or ``edges''.

The present tropical Fourier-Motzkin algorithm may be thought of as a
dual of the tropical double description
method~\cite{AllamigeonGaubertGoubaultDCG2013}, in which one
successively eliminates inequalities rather than variables.  In both
algorithms, redundant intermediate data (inequalities or generators) are
produced, and the key to the efficiency of the algorithm lies in the
dynamic elimination of such data. 
Redundant generators can be eliminated in almost linear time
using a combinatorial hypergraph algorithm, however 
the hypergraph criterion appears to have no natural dual analogue which
can detect redundant inequalities.

As mentioned above, the equivalence between mean payoff games and the
emptiness problem for tropical polyhedra with mixed constraints generalizes a result
of~\cite{AkianGaubertGutermanIJAC2011}. Moreover, it generalizes an earlier result~\cite{skutella} concerning finite solutions of
a class of non-strict disjunctive constraints appearing in scheduling.
The non-strict inequality satisfiability problem has also been studied
under the name of ``max-atom problem''
in~\cite{BezemNieuwenhuisCarbonnell08}, 
with motivations from SMT solving. Note here a fundamental difference
between strict and non-strict constraints: in the latter case, for
inequalities with integer coefficients, it is shown
in~\cite{AkianGaubertGutermanIJAC2011}
 that emptiness over
the integers is equivalent to emptiness over
the reals. The same is not true for strict inequalities (consider for example the open hypercube $]0,1[^n$ which is
non-empty, but contains no integer points), so that the
present result for mixed inequalities cannot be deduced from earlier
ones.

The infinitesimal perturbation of reals used in mixed inequalities is based on a semiring of affine germs, which was used in~\cite{GaubertGunawardena98,DhingraGaubertVALUETOOLS2006} to provide policy iteration based methods to solve mean payoff games. It also appeared 
in the context of tropical linear programming, see~\cite[\S~3.7]{GaubertKatzSergeevJSC2012}. The idea here
is that germs allow one to determine algebraically the value of a perturbed game. Related perturbation or parametric game ideas were used in~\cite{AllamigeonGaubertKatzLAA2011,SG10a}. 

\section{Tropical polyhedra with mixed constraints}\label{sec:mixed_constraints}

In the semiring $(\maxplus, \mpplus, \mptimes)$, addition and multiplication admit
neutral elements, namely $-\infty$ and $0$. 
Addition does not generally admit inverses. In contrast, any
non-zero (in the tropical sense) element $x$ admits a multiplicative inverse, 
which is given by $-x$ and will be denoted $x^{-1}$. 
The semiring operations are extended to vectors and
matrices in the usual way, \ie{} $(A \mpplus B)_{ij} \defi A_{ij} \mpplus B_{ij}$ and $(A
\mptimes B)_{ij} \defi \mpplus_k (A_{ik} \mptimes B_{kj})$. 
We will work in the semimodule $\maxplus^n$, for $n\in \Nat$.  Its
elements can be seen as points or vectors and are denoted $\vect x$, $\vect y$, \etc{}. 
The  order $\leq$ on $\maxplus$ is extended to vectors
entry-wise. We equip $\maxplus$ with the topology induced by the metric $(x, y) \mapsto \left|\exp x - \exp y \right|$, and $\maxplus^n$ with the product topology. In the sequel, we also use the completed max-plus semiring $\cmaxplus := \maxplus \cup \{+\infty\}$, with the conventions $x < +\infty$ for all $x \in \maxplus$, $x \mptimes (+\infty) = +\infty$ if $x \neq -\infty$, and $(-\infty) \mptimes (+\infty) =-\infty$.
Finally, given a positive integer $n$, we denote by $[n]$ the set $\{1, \dots, n\}$.

The notion of convexity can be transposed to tropical algebra. A subset 
$\CC$ of $\maxplus^n$ is said to be \emph{(tropically) convex} if it
contains the \emph{tropical segment} 
\[
\{ \lambda \vect{x} \mpplus \mu \vect{y}\mid \lambda , \mu \in \maxplus, \ \lambda \mpplus \mu = \mpone\}
\]
joining any two points $\vect{x}$ and $\vect{y}$ of $\CC$. This is analogous to the usual definition of
convexity, except that in the tropical setting the non-negativity
constraint on $\lambda$ and $\mu$ is implicit (any scalar $x \in
\maxplus$ satisfies $x \geq \mpzero$).

We now introduce the algebraic structure which will allow us to handle possibly strict tropical inequalities. We use a disjoint copy $\pert$ of $\R$ composed of elements denoted $\myul{\alpha}$ for $\alpha \in \R$, and we set $\smaxplus := \cmaxplus \cup \pert$. The \emph{modulus} $\abs{x}$ of an element $x \in \smaxplus$ is defined by:
\[
\abs{x} := 
\begin{cases}
x & \text{if}\ x \in \cmaxplus \; ; \\
\alpha & \text{if} \ x = \myul{\alpha} \in \pert \; .
\end{cases}
\]
The set $\smaxplus$ is totally ordered by the order relation $\sleq$ defined by: 
\[
x \sleq y \Longleftrightarrow 
\begin{cases}
\abs{x} < \abs{y} & \text{if} \ x \in \cmaxplus \ \text{and} \ y \in \pert \; ; \\
\abs{x} \leq \abs{y} & \text{otherwise.}
\end{cases}
\]
We use the notation $x \sless y$ when $x \sleq y$ and $x \neq y$. As an illustration, the Hasse diagram of $\sleq$ over the elements with modulus in $\mathbb{Z} \cup \{ \pm\infty\}$ is given in Figure~\ref{fig:hasse}.

\begin{figure}
\begin{center}
\begin{tikzpicture}
\node (infty) at (-2,3.5) {$+\infty$};
\node (n2) at (0,2) {$2$};
\node (n2p) at (2,1.5) {$\myul{2}$};
\node (n1) at (0,1) {$1$};
\node (n1p) at (2,0.5) {$\myul{1}$};
\node (n0) at (0,0) {$0$};
\node (n0p) at (2,-0.5) {$\myul{0}$};
\node (n-1) at (0,-1) {$-1$};
\node (n-1p) at (2,-1.5) {$\myul{(-1)}$};
\node (-infty) at (4,-3) {$-\infty$};

\node at (2,-2.25) {$\vdots$};
\node at (0,-1.75) {$\vdots$};
\node at (0,3) {$\vdots$};
\node at (2,2.5) {$\vdots$};

\draw (n-1p) -- (-infty) (n0p) -- (-infty) (n1p) -- (-infty) (n2p) -- (-infty);
\draw (n-1p) -- (n-1) -- (n0p) -- (n0) -- (n1p) -- (n1) -- (n2p) -- (n2);
\draw (n-1) -- (infty) (n0) -- (infty) (n1) -- (infty) (n2) -- (infty);
\end{tikzpicture}
\end{center}
\caption{Hasse diagram of the order $\sleq$ over the elements of $\smaxplus$ with modulus in $\mathbb{Z} \cup \{\pm\infty\}$.}\label{fig:hasse}
\end{figure}
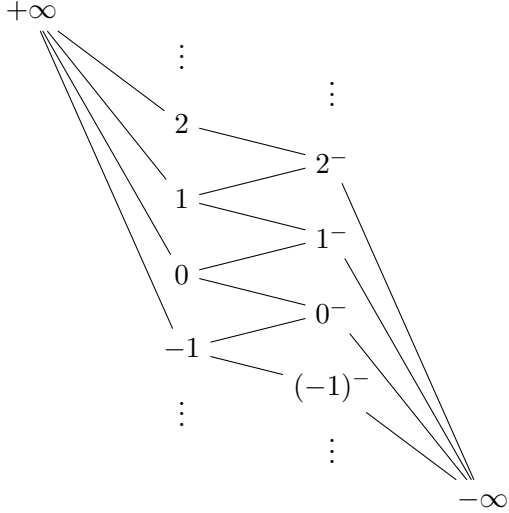
The element $\myul{\alpha}$ can be interpreted as an infinitesimal
perturbation of $\alpha$ of the form $\alpha - \eps$ with $\eps > 0$. Formally, given $x \in \smaxplus$ and $\eps > 0$, the valuation of $x$ at $\eps$, denoted by $x(\eps)$, is the element of $\cmaxplus$ defined as follows:
\[
x(\eps) := 
\begin{cases}
x & \text{if}\ x \in \cmaxplus \; ; \\
\abs{x} - \eps & \text{if} \ x \in \pert \; .
\end{cases}
\]
The valuation is extended to vectors and matrices entry-wise. 

The set $\smaxplus$ has a semiring structure when equipped with the sum of two elements $x, y \in \smaxplus$ defined as the greatest element among them, and the multiplication given by: 
\[
\begin{cases}
x \mptimes y & \text{if}\ x, y \in \R \; ;\\
\myul{(\abs{x} \mptimes \abs{y})} & 
\text{if} \ x \in \pert \ \text{or}\ y \in \pert, \ \text{and} \ x, y \neq \pm \infty\; ;\\
-\infty & \text{if} \ x = -\infty \ \text{or} \ y = -\infty \; ;\\
+\infty & \text{if} \ x, y \neq -\infty, \ \text{and} \ x = +\infty \ \text{or} \ y = +\infty \; .
\end{cases}
\]
By abuse of notation, the multiplication in $\smaxplus$ will be simply denoted by concatenation and the sum by $\mpplus$, as in the case of $\maxplus$. 
Observe that in $\smaxplus$ the neutral elements are still $-\infty$ and $0$,  
and that only the elements $x \in \R$ are invertible with respect to multiplication. Also note that the modulus map is a semiring morphism.

We begin with a technical lemma on the arithmetic operations in the semiring $\smaxplus$.
\begin{lemma}\label{lemma:smaxplus}
The following properties hold:
\begin{compactenum}[(i)]
\item\label{lemma:smaxplus:P3} for any $x, y \in \smaxplus$, $x \sleq y$ if, and only if, $x(\eps) \leq y(\eps)$ for $\eps > 0$ sufficiently small;
\item\label{lemma:smaxplus:P1} for any $x, y \in \smaxplus$ and $\eps > 0$ sufficiently small, $(x \mpplus y)(\eps) = x(\eps) \mpplus y(\eps)$;
\item\label{lemma:smaxplus:P2} for any $x \in \maxplus$, $y \in \smaxplus$, and $\eps > 0$, $(x y)(\eps) = x  y(\eps)$;
\item\label{lemma:smaxplus:P4} for any $x, y, z \in \smaxplus$, $x \sleq y$ implies $x z \sleq y z$, and the converse holds if $z \in \R$;
\item\label{lemma:smaxplus:P5} for any $x, y \in \maxplus$, $x < y$ is equivalent to $x \sleq \myul{0} y$ when $x \in \R$, and to $0 \sleq (+\infty)y$ when $x = -\infty$.
\end{compactenum}
\end{lemma}

\begin{proof}
\eqref{lemma:smaxplus:P3} The only non-trivial case is when $x \in \R$ and $y \in \pert$, so assume we are in this case. Then, $x \sleq y$ amounts to $x < \abs{y}$. This is equivalent to $x(\eps) = x \leq \abs{y} - \eps = y(\eps)$ for $\eps > 0$ sufficiently small.

\eqref{lemma:smaxplus:P1} Straightforward from Property~\eqref{lemma:smaxplus:P3}.

\eqref{lemma:smaxplus:P2} This property readily follows from the definition of the multiplication.

\eqref{lemma:smaxplus:P4} If $z = \pm \infty$, the first property is straightforward, so assume $\abs{z}\in \R$. Let $x \sless y$ (the implication is trivial when $x=y$). If $\abs{x} < \abs{y}$, we have $\abs{x z} < \abs{y z}$ because $z \neq  \pm \infty $. Moreover, if $\abs{x} = \abs{y} =: \alpha$, then $\alpha \neq \pm \infty$ (as $x$ and $y$ are distinct), and thus necessarily $x = \myul{\alpha}$ and $y = \alpha$ (because $x \sless y$). Hence $x z\in \pert$ and $\abs{x z} = \abs{y z}$. In both cases, we conclude that $x z \sleq y z$.

Conversely, assume that $z \in \R$. Using the first part of the proof, $x z \sleq y z$ implies $x \sleq y$ by multiplying both sides of the inequality $x z \sleq y z$ by $z^{-1}$.

\eqref{lemma:smaxplus:P5} In the first place, we suppose that $x \in \R$. If $x < y$, then $y \in \R$ and $\myul{0} y  = \myul{y}$. Thus, $x < y$ implies $x \sleq \myul{y}$, \ie\ $x \sleq \myul{0}y$. Conversely, if the latter inequality holds, then $y$ is distinct from $-\infty$. Thus, $\myul{0}y = \myul{y}$, and $x \sleq \myul{y}$ ensures that $x < y$.

Assume now that $x=\mpzero$. Note that $(+\infty) y$ is equal to $+\infty$ if $y \neq -\infty$, and to $-\infty$ otherwise. Thus, we have $0 \sleq (+\infty)y$ if, and only if, $y > -\infty=x$.
\end{proof}

A \emph{mixed tropical affine inequality} 
is defined as a constraint of the form
\begin{equation}\label{eq:mixed_affine_inequality}
a_0 \mpplus a_1 \vect{x}_1 \mpplus \cdots \mpplus a_n \vect{x}_n \sleq b_0 \mpplus b_1 \vect{x}_1 \mpplus \cdots \mpplus b_n \vect{x}_n\enspace, 
\end{equation}
where the coefficients $a_i$ on the left-hand side belong $\maxplus$,
while the coefficients $b_i$ on the right-hand side are in
$\smaxplus$. When the set of $\vect{x}$ in $\maxplus^n$ satisfying a mixed inequality is a non-empty
proper subset of $\maxplus^n$, it is called \emph{mixed half-space}. 
\begin{lemma}\label{lemma:val}
A vector $\vect{x} \in \maxplus^n$ 
satisfies~\eqref{eq:mixed_affine_inequality} if, 
and only if, there exists $\eps > 0$ such that
\begin{equation}
a_0 \mpplus a_1 \vect{x}_1 \mpplus \cdots \mpplus a_n \vect{x}_n \leq b_0(\eps) \mpplus b_1(\eps) \vect{x}_1 \mpplus \cdots \mpplus b_n(\eps) \vect{x}_n \enspace. \label{eq:mixed_affine_inequality_val}
\end{equation}
\end{lemma}

\begin{proof}
If~\eqref{eq:mixed_affine_inequality} is satisfied, then by Property~\eqref{lemma:smaxplus:P3} of Lemma~\ref{lemma:smaxplus} we have
\[
a_0 \mpplus a_1 \vect{x}_1 \mpplus \cdots \mpplus a_n \vect{x}_n \leq (b_0 \mpplus b_1 \vect{x}_1 \mpplus \cdots \mpplus b_n \vect{x}_n)(\eps)  
\]
for $\eps > 0$ sufficiently small, since the left-hand side of~\eqref{eq:mixed_affine_inequality} belongs to $\maxplus^n$. Besides, by Properties~\eqref{lemma:smaxplus:P1} and~\eqref{lemma:smaxplus:P2} of Lemma~\ref{lemma:smaxplus}, it follows that 
\[
(b_0 \mpplus b_1 \vect{x}_1 \mpplus \cdots \mpplus b_n \vect{x}_n)(\eps) = b_0(\eps) \mpplus b_1(\eps) \vect{x}_1 \mpplus \cdots \mpplus b_n(\eps) \vect{x}_n 
\]
for $\eps > 0$ sufficiently small, which shows that~\eqref{eq:mixed_affine_inequality_val} holds. 

Conversely, suppose that~\eqref{eq:mixed_affine_inequality_val} is satisfied for some $\eps > 0$. Then 
\[
a_0 \mpplus a_1 \vect{x}_1 \mpplus \cdots \mpplus a_n \vect{x}_n \leq b_0(\eps') \mpplus b_1(\eps') \vect{x}_1 \mpplus \cdots \mpplus b_n(\eps') \vect{x}_n
\]
for any $\eps' < \eps$ (the map $\eps' \mapsto b_0(\eps') \mpplus b_1(\eps') \vect{x}_1 \mpplus \cdots \mpplus b_n(\eps') \vect{x}_n$ is non-increasing). It follows that~\eqref{eq:mixed_affine_inequality} holds, by Properties~\eqref{lemma:smaxplus:P3}, \eqref{lemma:smaxplus:P1} and~\eqref{lemma:smaxplus:P2} of Lemma~\ref{lemma:smaxplus} and the fact that the left-hand side of~\eqref{eq:mixed_affine_inequality} belongs to $\maxplus$.
\end{proof}

A \emph{tropical polyhedron with mixed constraints} is defined as a set composed of the vectors $\vect{x} \in \maxplus^n$ which satisfy finitely many mixed tropical affine inequalities. 
To contrast with, we use the term \emph{closed tropical polyhedron} when the defining mixed inequalities only involve coefficients in $\maxplus$, \ie\ they are of the form~\eqref{eq:affine_inequality}.
The following proposition establishes that polyhedra with mixed constraints are (possibly non-closed) tropically convex sets. 

\begin{proposition}\label{prop:mixed_polyhedra_convex}
Any tropical polyhedron with mixed constraints is a tropically convex set.
\end{proposition}

\begin{proof}
Let $\vect{x},\vect{y}$ be two solutions of~\eqref{eq:mixed_affine_inequality}, 
and $\lambda, \mu \in \maxplus$ be such that $\lambda \mpplus \mu = \mpone$. 
By Lemma~\ref{lemma:val}, there exist $\eps, \eps' > 0$ such that:
\begin{align*}
a_0 \mpplus a_1 \vect{x}_1 \mpplus \cdots \mpplus a_n \vect{x}_n & \leq b_0(\eps) \mpplus b_1(\eps) \vect{x}_1 \mpplus \cdots \mpplus b_n(\eps) \vect{x}_n \enspace, \\
a_0 \mpplus a_1 \vect{y}_1 \mpplus \cdots \mpplus a_n \vect{y}_n & \leq b_0(\eps') \mpplus b_1(\eps') \vect{y}_1 \mpplus \cdots \mpplus b_n(\eps') \vect{y}_n \enspace.
\end{align*}
These inequalities are still valid if we replace $\eps$ and $\eps'$ by
$\min(\eps, \eps')$. Hence, we can assume, without loss of generality,
that $\eps = \eps'$.  Then, $\vect{z} = \lambda \vect{x} \mpplus \mu
\vect{y}$ satisfies
\[
a_0 \mpplus a_1 \vect{z}_1 \mpplus \cdots \mpplus a_n \vect{z}_n \leq b_0(\eps) \mpplus b_1(\eps) \vect{z}_1 \mpplus \cdots \mpplus b_n(\eps) \vect{z}_n \enspace, \\
\]
which proves that $\vect{z}$ is a solution 
of~\eqref{eq:mixed_affine_inequality} by Lemma~\ref{lemma:val}. 
Thus, any mixed half-space is tropically convex. 
We conclude that every tropical polyhedron with mixed constraints is
tropically convex, as the intersection of finitely many tropically
convex sets. 
\end{proof}

\begin{example}
  The vectors $\vect{x} \in \Real^2$ satisfying the strict inequality
  $\vect{x}_1 < \max(-1+\vect{x}_2, 0)$, depicted on the left-hand side of
  Figure~\ref{fig:mixed}, are obtained as the real solutions of
  the mixed affine inequality $\vect{x}_1 \sleq \myul{(-1)} \vect{x}_2
  \mpplus \myul{0}$. Similarly, the solutions of
  $\vect{x}_1 \sleq \myul{(-1)} \vect{x}_2 \mpplus 0$ correspond to the
  previous set in which the half-line $\{(0,\lambda) \mid \lambda \leq
  1\}$ is added (middle of Figure~\ref{fig:mixed}). 

  The set depicted on the right-hand side of Figure~\ref{fig:mixed} is
  the tropical polyhedron with mixed constraints defined by the
  following mixed inequalities:
  \begin{equation}
    \begin{aligned}
      (-2)\vect{x}_2 & \sleq \myul{0} \mpplus \myul{0} \vect{x}_1 \\
      -3 & \sleq \vect{x}_1 \\
      0 & \sleq 1\vect{x}_1 \mpplus \myul{0}\vect{x}_2
    \end{aligned}
    \qquad\qquad
    \begin{aligned}
      -2 & \sleq \vect{x}_2 \\
      \vect{x}_1 & \sleq \myul{3} \vect{x}_2 \\
      (-2)\vect{x}_1 & \sleq \myul{0} \mpplus (-1) \vect{x}_2
    \end{aligned}
    \label{eq:running_example}
  \end{equation}
\end{example}

\begin{figure}[tp]
\centering
\begin{tikzpicture}[>=stealth',convex/.style={draw=lightgray,fill=lightgray,fill opacity=0.7},point/.style={blue!50},line/.style={blue!50,ultra thick},convexborder/.style={ultra thick},scale=1.15]

\begin{scope}[scale=0.52]
\draw[gray!40,very thin] (-3.5,-2.5) grid (3.5,4.5);
\draw[gray!80,->] (-3.5,0) -- (3.5,0) node[color=gray!80,above] {$\vect{x}_1$};
\draw[gray!80,->] (0,-2.5) -- (0,5.1) node[color=gray!80,right] {$\vect{x}_2$};
\filldraw[convex] (-3.5,-2.5) -- (0,-2.5) -- (0,1) -- (3.5,4.5) -- (-3.5,4.5) -- cycle;
\end{scope}

\begin{scope}[scale=0.52,xshift=9.5cm]
\draw[gray!40,very thin] (-3.5,-2.5) grid (3.5,4.5);
\draw[gray!80,->] (-3.5,0) -- (3.5,0) node[color=gray!80,above] {$\vect{x}_1$};
\draw[gray!80,->] (0,-2.5) -- (0,5.1) node[color=gray!80,right] {$\vect{x}_2$};
\filldraw[convex] (-3.5,-2.5) -- (0,-2.5) -- (0,1) -- (3.5,4.5) -- (-3.5,4.5) -- cycle;
\draw[convexborder] (0,-2.5) -- (0,1) circle (1.5pt);
\end{scope}

\begin{scope}[scale=0.52,xshift=19cm]
\draw[gray!40,very thin] (-3.5,-2.5) grid (3.5,4.5);
\draw[gray!80,->] (-3.5,0) -- (3.5,0) node[color=gray!80,above] {$\vect{x}_1$};
\draw[gray!80,->] (0,-2.5) -- (0,5.1) node[color=gray!80,right] {$\vect{x}_2$};
\filldraw[convex] (-1,0) -- (-1,-2) -- (1,-2) -- (2,-1) -- (2,1) -- (3.5, 2.5) -- (3.5,4.5) -- (2.5,4.5) -- (0,2) -- (-3,2) -- (-3,0) -- cycle;
\draw[convexborder] (-1,0) circle (1.5pt) -- (-1,-2) -- (1,-2) (2,1) circle (1.5pt) -- (3.5, 2.5) (-3,2) -- (-3,0) ;  
\end{scope}
\end{tikzpicture}
\caption{Tropical polyhedra with mixed constraints (the ends of the
  black segments marked by points are included in the
  polyhedra).}\label{fig:mixed}
\end{figure}
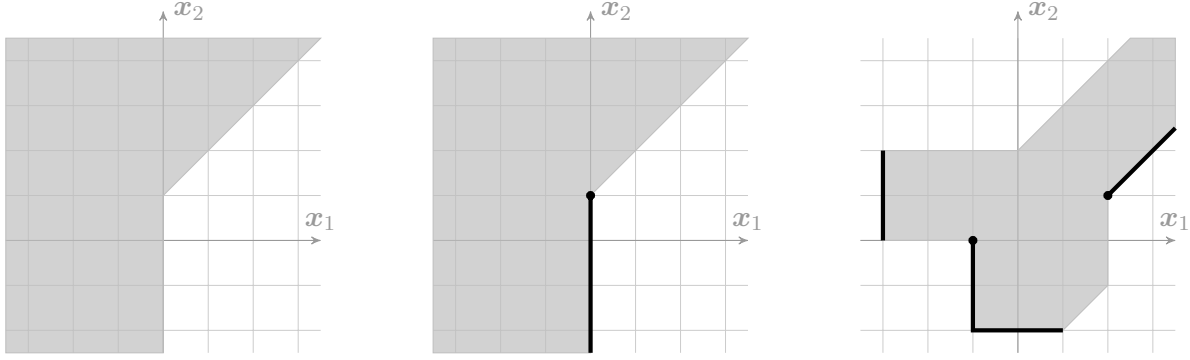

Observe that the inequalities in~\eqref{eq:running_example} have the 
property that no variable (or constant term) ever appears on both the
left- and right-hand sides. The following lemma ensures that this
situation is not restrictive, and that in any inequality of the form~\eqref{eq:mixed_affine_inequality}, we can always assume $a_i=
\mpzero$ or $b_i = -\infty$ for all $i \in \{0, \dots, n\}$.

\begin{lemma}\label{lemma:disjoint_support}
  Let $a, b \in \maxplus$ and $c, d \in \smaxplus$. The set of solutions in $\maxplus$ of the
  mixed tropical affine inequality $a x \mpplus b \sleq c x \mpplus d$
  is given by $\{ x \in \maxplus \mid b \sleq c x \mpplus d \}$ if $a
  \sleq c$, and by $\{ x \in \maxplus \mid a x \mpplus b \sleq d \}$
  otherwise.
\end{lemma}

\begin{proof}
In the first place, suppose that $a \sleq c$. Then, 
$a x \sleq c x$ for all $x \in \maxplus$ by 
Property~\eqref{lemma:smaxplus:P4} of Lemma~\ref{lemma:smaxplus}. 
It follows that $a x \sleq c x \mpplus d$ is always satisfied.

Suppose now that $a \sgreater c$. If $x \in \R$, 
we have $a x \mpplus b \sgeq a x \sgreater c x$ 
by Property~\eqref{lemma:smaxplus:P4} of Lemma~\ref{lemma:smaxplus}. 
As a consequence, any $x \in \maxplus$ such that 
$a x \mpplus b \sleq c x \mpplus d$ also satisfies $a x \mpplus b \sleq d$. 
This completes the proof. 
\end{proof}

\begin{remark}\label{remark:generators}
Following the analogy with closed tropical polyhedra, we could also 
consider subsets of $\smaxplus^n$ generated by finitely many points 
and rays (given by vectors with entries in $\smaxplus$), having in mind to encode 
non-closed subsets of $\maxplus^n$ thanks to infinitesimal perturbation 
of generators. More precisely, we could consider that a subset $\CC$ of 
$\smaxplus^n$ encodes the subset $\widetilde{\CC}$ of $\maxplus^n$ 
given by the points $\vect{y} \in \maxplus^n$ such that for all $\eps > 0$ 
there exist $\eps'\in \left] 0,\eps\right[$ and $\vect{x} \in \CC$ verifying $\vect{y} = \vect{x}(\eps')$.
However, this approach presents several 
difficulties. For example, consider the segment $\mcS$ of $\smaxplus^2$ 
joining the points $\vect{v} = (\myul{0},1)$ and $\vect{w} = (\myul{1},1)$, \ie  
\[
\mcS = \{\lambda \vect{v} \mpplus \mu \vect{w} \in \smaxplus^2 \mid \lambda, \mu \in \smaxplus, \ \lambda \mpplus \mu = \mpone\}\; . 
\]
Then, it can be checked that 
$\widetilde{\mcS} = \{(\alpha, 1)\in \maxplus^2  \mid \alpha \in \left[ 0, 1 \right[ \}$. 
Now, observe that if we consider the segment $\mcS'$ joining $\vect{v}'=(\myul{1},0)$ and $\vect{w}'=(0,1)$, we have $\widetilde{\mcS'}=\widetilde{\mcS}$. 
Noticing that $\vect{v}' \not \in \mcS$ and $\vect{v} \not \in \mcS'$,
we see that different polyhedra of $\smaxplus^n$ can encode the same 
subset of $\maxplus^n$. Consequently, it seems far from trivial to determine 
the equality of two subsets of $\maxplus^n$ when they are encoded as 
polyhedra of $\smaxplus^n$.
\end{remark}

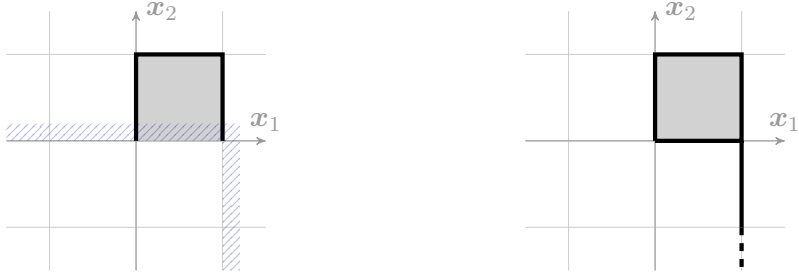
\begin{figure}
\centering
\begin{tikzpicture}[>=stealth',convex/.style={draw=lightgray,fill=lightgray,fill opacity=0.7},point/.style={blue!50},line/.style={blue!50,ultra thick},convexborder/.style={ultra thick},scale=1.15]

\begin{scope}[hs/.style={draw=none,pattern=north east lines,pattern color=blue!60!black,fill opacity=0.5}]
\draw[gray!40,very thin] (-1.5,-1.5) grid (1.5,1.5);
\draw[gray!80,->] (-1.5,0) -- (1.5,0) node[color=gray!80,above] {$\vect{x}_1$};
\draw[gray!80,->] (0,-1.5) -- (0,1.5) node[color=gray!80,right] {$\vect{x}_2$};
\filldraw[convex] (0,0) -- (0,1) -- (1,1) -- (1,0) -- cycle;
\draw[convexborder] (0,0) -- (0,1) -- (1,1) -- (1,0);
\filldraw[hs] (-1.5,0) -- (1,0) -- (1,-1.5) -- (1.2,-1.5) -- (1.2,0.2) -- (-1.5,0.2) -- cycle;
\end{scope}

\begin{scope}[xshift=6cm]
\draw[gray!40,very thin] (-1.5,-1.5) grid (1.5,1.5);
\draw[gray!80,->] (-1.5,0) -- (1.5,0) node[color=gray!80,above] {$\vect{x}_1$};
\draw[gray!80,->] (0,-1.5) -- (0,1.5) node[color=gray!80,right] {$\vect{x}_2$};
\filldraw[convex] (0,0) -- (0,1) -- (1,1) -- (1,0) -- cycle;
\draw[convexborder] (0,0) -- (0,1) -- (1,1) -- (1,0) -- (0,0) (1,0) -- (1,-1);
\draw[convexborder,dashed] (1,-1) -- (1,-1.5);
\end{scope}

\end{tikzpicture}
\caption{Left: a tropical polyhedron with mixed constraints, together with an open mixed half-space (in blue) defining it. Right: the closed polyhedron defined by the corresponding non-strict inequalities.}\label{fig:closure}
\end{figure}

\begin{remark}
As a complement of Remark~\ref{remark:generators}, we point out that the closure of a polyhedron with mixed constraints is apparently harder to compute than in the case of usual convex polyhedra. Indeed, in the latter case, the closure can be simply obtained by substituting $<$ by $\leq$ in the defining inequalities. In contrast, in the current setting, replacing the coefficients $b_i$ of the form $\myul{\beta_i}$ by $\beta_i$ on the right-hand side of mixed inequalities provides a closed tropical polyhedron which may be larger than the closure. For example, consider the polyhedron with mixed constraints defined by the inequalities:
\[
0 \sleq \vect{x}_1 \qquad \vect{x}_1 \mpplus \vect{x}_2 \sleq 1 \qquad 0 \sleq \myul{(-1)} \vect{x}_1 \mpplus \myul{0} \vect{x}_2 \; ,
\]
which is depicted on the left-hand side of Figure~\ref{fig:closure}. The mixed half-space defined by the last inequality is represented in blue. The closure of the polyhedron is the usual unit square. However, the closed polyhedron defined by the inequalities:
\[
0 \leq \vect{x}_1 \qquad \vect{x}_1 \mpplus \vect{x}_2 \leq 1 \qquad 0 \leq (-1) \vect{x}_1 \mpplus \vect{x}_2 \; ,
\]
contains additionally the half-line $\{ (1, \lambda) \mid \lambda \leq 0 \}$ (right-hand side of Figure~\ref{fig:closure}). 
\end{remark}

\section{Tropical Fourier-Motzkin elimination}\label{sec:fourier_motzkin}

In this section, we first present a tropical Fourier-Motzkin elimination method, which allows to eliminate a variable in a finite system of mixed inequalities. Then we apply this method to establish relationships between tropical polyhedra with mixed constraints, zones and tropical hemispaces.

\subsection{The algorithm}\label{subsec:elementary_step}

We first illustrate the algorithm on an example.

\begin{example}\label{ex:fourier_motzkin}
Consider the system given in~\eqref{eq:running_example}, and assume we want to eliminate the
variable $\vect{x}_1$. From the last (rightmost) two inequalities
of~\eqref{eq:running_example}, we know that:
\begin{equation}
\vect{x}_1 \sleq \myul{3} \vect{x}_2 \qquad \vect{x}_1 \sleq \myul{2} \mpplus 1 \vect{x}_2 \label{eq:fm_example1}
\end{equation}
In each inequality involving the variable $\vect{x}_1$ on the right-hand side (\ie{} the leftmost three inequalities of~\eqref{eq:running_example}), we propose to replace $\vect{x}_1$ by the two upper bounds provided by~\eqref{eq:fm_example1}. This produces the following six inequalities:
\begin{equation}
\begin{aligned} 	
(-2) \vect{x}_2 & \sleq \myul{0} \mpplus \myul{3} \vect{x}_2 \\
(-2) \vect{x}_2 & \sleq \myul{2} \mpplus \myul{1} \vect{x}_2
\end{aligned}
\qquad
\begin{aligned} 
-3 & \sleq \myul{3} \vect{x}_2\\
-3 & \sleq \myul{2} \mpplus 1 \vect{x}_2 
\end{aligned}
\qquad
\begin{aligned}
0 & \sleq \myul{4} \vect{x}_2 \\
0 & \sleq \myul{3} \mpplus 2 \vect{x}_2 
\end{aligned}
\label{eq:fm_example2}
\end{equation}
Besides, in each inequality not involving $\vect{x}_1$ on the right-hand side, we remove the term in $\vect{x}_1$ from the left-hand side, if any. From the rightmost inequalities of~\eqref{eq:running_example}, we obtain the following three inequalities:
\begin{equation}
-2 \sleq \vect{x}_2 \qquad \mpzero \sleq \myul{3} \vect{x}_2 \qquad \mpzero \sleq \myul{0} \mpplus (-1) \vect{x}_2 \enspace. \label{eq:fm_example3}
\end{equation}
We claim that the inequalities in~\eqref{eq:fm_example2} and~\eqref{eq:fm_example3} precisely describe the projection on the $\vect{x}_2$ axis of the polyhedron with mixed constraints defined by~\eqref{eq:running_example}. Note that the collection of inequalities obtained in this way is redundant: one inequality, $-2\sleq \vect{x}_2$, suffices.
\end{example}
We now formalize the approach sketched in Example~\ref{ex:fourier_motzkin}. 
Under some assumptions (which are specified in Theorem~\ref{th:fourier_motzkin} below), it applies more generally to systems of constraints over a totally ordered idempotent semiring. 

Recall that a semiring $(\RR,\mpplus,\mptimes,\RRmpzero,\RRmpone)$ is said to be \emph{totally ordered} if there exists a total order $\RRleq$ on $\RR$ such that: 
\begin{enumerate}[(i)]
\item $\RRmpzero\RRleq a$ for all $a\in \RR$,
\item for all $a,b,c \in \RR$, $a \RRleq b$ implies $a \mpplus c \RRleq b \mpplus c$, $a \mptimes c \RRleq b \mptimes c$, and $c \mptimes a \RRleq c \mptimes b$.
\end{enumerate}
The semiring $\RR$ is said to be \emph{idempotent} 
if $a \mpplus a = a$ for all $a \in \RR$.  
The next lemma shows that such a semiring is \emph{naturally ordered}, 
meaning that $a \mpplus b$ is equal to the maximal element among $a$ and $b$.

\begin{lemma}\label{lemma:totally_ordered_idempotent}
Let $(\RR,\mpplus,\mptimes,\RRmpzero,\RRmpone,\RRleq)$ be a totally ordered 
idempotent semiring. Then, for all $a, b \in \RR$, $a \mpplus b = a$ if 
$a \RRgeq b$, and $a \mpplus b = b$ otherwise.
\end{lemma}

\begin{proof}
First, observe that $b \RRgeq \RRmpzero$ implies $a \mpplus b \RRgeq a$. 
Analogously, we have $a \mpplus b \RRgeq b$. Now suppose that $a \RRgeq
b$. Then, $a = a \mpplus b$ since $a = a \mpplus a \RRgeq a \mpplus b$. 
\end{proof}

Now we explain how to eliminate $\vect{x}_n$ in a linear system of inequalities over $\RR$ in the variables $\vect{x}_1, \ldots ,\vect{x}_n$. For the sake of simplicity, we extend the operations of $\RR$ to matrices and vectors in the usual way and represent the multiplication $\mptimes$ by concatenation. 

\begin{theorem}[Fourier-Motzkin elimination for systems over totally ordered idempotent semirings]\label{th:fourier_motzkin}
Let $(\RR,\mpplus,\mptimes,\RRmpzero,\RRmpone,\RRleq)$ be a totally ordered idempotent semiring and $\PP \subset \RR^n$ be the solution set of the system 
$A \vect{x} \mpplus \vect{c} \RRleq B \vect{x} \mpplus \vect{d}$, 
where $A, B \in \RR^{p \times n}$ and $\vect{c}, \vect{d} \in \RR^p$ satisfy the following conditions: 
\begin{enumerate}[(i)]
\item\label{item:cond1} $A_{in}$ is (left-)invertible with respect to $\mptimes$ if $A_{in} \neq \RRmpzero$ (we denote its inverse by $A_{in}^{-1}$),\footnote{Here and below, $M_{in}$ denotes the $(i,n)$-entry of matrix $M$, and should not be confused with any abbreviation.} 
\item\label{item:cond2} for any $\alpha\in \RR$ there exists $\beta\in \RR$ such that $\alpha \RRleq B_{in} \beta$ if $B_{in} \neq \RRmpzero$, 
\item\label{item:cond3} either $A_{ij} = \RRmpzero$ or $B_{ij} = \RRmpzero$ for $i\in \oneto{p}$ and $j\in \oneto{n}$. 
\end{enumerate}
Let $\QQ \subset \RR^{n-1}$ be the set defined by the following inequalities in the variables $\vect{x}_1,\ldots , \vect{x}_{n-1}$:
\begin{equation}\label{eq:fm1}
(\mpplus_{j \neq n} A_{ij} \vect{x}_j ) \mpplus \vect{c}_i \RRleq ( \mpplus_{j \neq n} B_{ij} \vect{x}_j ) \mpplus \vect{d}_i \enspace, 
\end{equation}
for all $i \in [p]$ such that $B_{in} = \RRmpzero$, and 
\begin{equation}\label{eq:fm2}
(\mpplus_{j \neq n} A_{ij} \vect{x}_j ) \mpplus \vect{c}_i \RRleq ( \mpplus_{j \neq n} (B_{ij} \mpplus B_{in} A_{kn}^{-1} B_{kj}) \vect{x}_j ) \mpplus \vect{d}_i \mpplus B_{in} A_{kn}^{-1} \vect{d}_k \enspace, 
\end{equation}
for all $i, k \in [p]$ such that $B_{in} \neq \RRmpzero$ and $A_{kn} \neq \RRmpzero$.

Then $\vect{x} \in \QQ$ if, and only if, there exists $\lambda \in \RR$ such that $(\vect{x}, \lambda) \in \PP$.
\end{theorem}

\begin{proof}
Assume that $(\vect{x}, \lambda) \in \PP$ for some $\lambda \in \RR$. Then for all $i \in [p]$ such that $B_{in} = \RRmpzero$, we have:
\[
(\mpplus_{j \neq n} A_{ij}\vect{x}_j ) \mpplus \vect{c}_i \RRleq ( \mpplus_{j \neq n} A_{ij} \vect{x}_j ) \mpplus A_{in} \lambda \mpplus \vect{c}_i \RRleq ( \mpplus_{j \neq n} B_{ij} \vect{x}_j ) \mpplus \vect{d}_i \enspace,
\]
which proves that $\vect{x}$ satisfies the inequalities of the form~\eqref{eq:fm1}. Now consider $i, k \in [p]$ such that $B_{in} \neq \RRmpzero$ and $A_{kn} \neq \RRmpzero$. Then, $A_{in}=B_{kn}=\RRmpzero$ by Condition~\eqref{item:cond3}, and $A_{kn}$ is left-invertible by Condition~\eqref{item:cond1}. Since $A_{kn} \lambda \RRleq ( \mpplus_{j \neq n} B_{kj} \vect{x}_j ) \mpplus \vect{d}_k$, we know that 
\[
\lambda \RRleq A_{kn}^{-1} (\mpplus_{j \neq n} B_{kj} \vect{x}_j ) \mpplus A_{kn}^{-1} \vect{d}_k  \enspace.
\]
Replacing $\lambda$ by the latter upper bound in the inequality 
\[
( \mpplus_{j \neq n} A_{ij}\vect{x}_j ) \mpplus \vect{c}_i = ( \mpplus_{j \neq n}  A_{ij}\vect{x}_j ) \mpplus A_{in} \lambda  \mpplus \vect{c}_i \RRleq ( \mpplus_{j \neq n} B_{ij} \vect{x}_j ) \mpplus B_{in} \lambda \mpplus \vect{d}_i
\]
precisely yields inequality~\eqref{eq:fm2}, using distributivity of $\mptimes$ and commutativity of $\mpplus$. Thus, we conclude that $\vect{x} \in \QQ$. 

Conversely, let $\vect{x} \in \QQ$. In the first place, 
assume that $A_{in} = \RRmpzero$ for all $i\in \oneto{p}$. Define   
\[
\lambda \defi \mpplus_{i \in I} \beta_i  \enspace ,
\]
where $I\mydef \{i\in \oneto{p} \mid B_{in} \neq \RRmpzero\}$ and $\beta_i \in \RR$ is such that $(\mpplus_{j \neq n} A_{ij} \vect{x}_j ) \mpplus \vect{c}_i \RRleq B_{in} \beta_i$ for $i \in I$ ($\beta_i$ exists thanks to Condition~\eqref{item:cond2}). Then, if $B_{in} \neq \RRmpzero$, we have 
\[
( \mpplus_{j\neq n} A_{ij}\vect{x}_j ) \mpplus A_{in} \lambda \mpplus \vect{c}_i = ( \mpplus_{j \neq n} A_{ij}\vect{x}_j ) \mpplus \vect{c}_i \RRleq B_{in} \lambda \RRleq ( \mpplus_{j \neq n} B_{ij} \vect{x}_j ) \mpplus B_{in} \lambda \mpplus \vect{d}_i \; .
\]
The same inequality is trivially satisfied if $B_{in} = \RRmpzero$ due to~\eqref{eq:fm1}. It follows that $(\vect{x},\lambda) \in \PP$. 

Now assume that $A_{kn} \neq \RRmpzero$ for some $k\in \oneto{p}$. Define 
\begin{equation}\label{eq:fm_lambda}
\lambda \defi \min_{\substack{k \in \oneto{p}\\ A_{kn} \neq \RRmpzero}} \bigl(A_{kn}^{-1} (\mpplus_{j \neq n} B_{kj} \vect{x}_j ) \mpplus A_{kn}^{-1} \vect{d}_k\bigr)  \enspace, 
\end{equation}
where the operator $\min$ is understood as providing the minimum of its operands with respect to the order $\RRleq$. As a consequence, for all $i$ such that $A_{in} \neq \RRmpzero$, we have $A_{in} \lambda \RRleq ( \mpplus_{j \neq n} B_{ij} \vect{x}_j ) \mpplus \vect{d}_i$. The fact that $B_{in}= \RRmpzero$ (by Condition~\eqref{item:cond3}) and the conjunction with~\eqref{eq:fm1} yield:
\[
(\mpplus_{j \neq n} A_{ij} \vect{x}_j ) \mpplus A_{in} \lambda \mpplus \vect{c}_i \RRleq ( \mpplus_{j \neq n} B_{ij} \vect{x}_j ) \mpplus \vect{d}_i =
(\mpplus_{j \neq n} B_{ij} \vect{x_j} ) \mpplus B_{in} \lambda \mpplus \vect{d_i}
\enspace .
\]
Note that the latter inequality also holds for all $i\in [p]$ such that $A_{in}$ and $B_{in}$ are both equal to $\RRmpzero$. 

Now suppose that $i \in [p]$ satisfies $B_{in} \neq \RRmpzero$ and $k \in [p]$ attains the minimum in~\eqref{eq:fm_lambda}. Since $\vect{x}$ satisfies~\eqref{eq:fm2}, \ie 
\[
(\mpplus_{j \neq n} A_{ij} \vect{x}_j ) \mpplus \vect{c}_i \RRleq ( \mpplus_{j \neq n} B_{ij} \vect{x_j} ) \mpplus \vect{d_i} \mpplus B_{in} A_{kn}^{-1} ( \mpplus_{j \neq n} B_{kj} \vect{x}_j \mpplus \vect{d}_k )
\]
it follows that
\[
(\mpplus_{j \neq n} A_{ij} \vect{x}_j )\mpplus A_{in} \lambda \mpplus \vect{c}_i =
(\mpplus_{j \neq n} A_{ij} \vect{x}_j ) \mpplus \vect{c}_i \RRleq ( \mpplus_{j \neq n} B_{ij} \vect{x_j} ) \mpplus B_{in} \lambda \mpplus \vect{d_i} \enspace ,
\]
because $B_{in}\neq \RRmpzero$ implies $A_{in}= \RRmpzero$ by Condition~\eqref{item:cond3}. This shows that $(\vect{x},\lambda)$ belongs to $\PP$. 
\end{proof}

The case of tropical polyhedra with mixed constraints is obtained by
setting $\RR = \smaxplus$. Note that the conditions of Theorem~\ref{th:fourier_motzkin} 
are satisfied when $\RR = \smaxplus$, due in particular to Lemma~\ref{lemma:disjoint_support} and the fact that any non-zero coefficient on the
left-hand side of a mixed inequality is invertible (as an element of
$\R$). However, a tropical polyhedron with
mixed constraints consists of the solutions belonging to $\maxplus^n$,
while Theorem~\ref{th:fourier_motzkin} applies to the solutions in
$\smaxplus^n$. The following result shows that the projection algorithm
is still valid when restricted to $\maxplus^n$.

\begin{theorem}[Fourier-Motzkin elimination for systems of mixed inequalities]\label{th:fourier_motzkin_maxplus_case}
  Assume $\PP$ and $\QQ$ are defined as in Theorem~\ref{th:fourier_motzkin} with $\RR=\smaxplus$. Then, for all $\vect{x} \in \maxplus^{n-1}$, $\vect{x} \in \QQ$ if, and only if, there exists $\lambda \in
  \maxplus$ such that $(\vect{x},\lambda) \in \PP$.
\end{theorem}

\begin{proof}
Observe that to prove the theorem, it is enough to show that 
if $\vect{y} \mydef (\vect{x},\lambda) \in \PP$ for some 
$\vect{x} \in \maxplus^{n-1}$ and $\lambda \in \smaxplus$, then there exists 
$\lambda' \in \maxplus$ such that $(\vect{x},\lambda') \in \PP$.

If $\lambda \in \maxplus$, there is nothing to prove, so assume $\lambda \in \pert \cup \{+\infty\}$. We will show that for a certain choice of $\lambda' \in \R$ verifying $\lambda' \sleq \lambda$, the vector $\vect{y}' := (\vect{x}, \lambda') \in \maxplus^n$ belongs to $\PP$. Note that for any such choice of $\lambda'$, the vector $\vect{y}'$ satisfies the inequalities in $A \vect{x} \mpplus \vect{c} \sleq B \vect{x} \mpplus \vect{d}$ indexed by $i \in \oneto{p}$ such that $B_{in} = \pm \infty$. Indeed, in this case, we have
\begin{equation}
(\mpplus_{j} A_{ij} \vect{y}'_j ) \mpplus \vect{c}_i \sleq ( \mpplus_{j} A_{ij} \vect{y}_j ) \mpplus \vect{c}_i \sleq ( \mpplus_{j} B_{ij} \vect{y}_j ) \mpplus \vect{d}_i = ( \mpplus_{j} B_{ij} \vect{y}'_j ) \mpplus \vect{d}_i \; . \label{eq:common_case}
\end{equation}
Thus, we next focus on the inequalities indexed by elements of the set $I := \{ i \in \oneto{p} \mid B_{in} \in \R \cup \pert \}$. In consequence, $A_{in} = \mpzero$ for all $i\in I$.
It is convenient to split the rest of the proof into two cases. 

In the first place, assume $\lambda = +\infty$.
Define $\lambda' :=\delta \epsilon \in \R$, where $\epsilon > 0$ and 
\[
\delta =
\begin{cases}
\mpplus_{i\in I } \abs{B_{in}}^{-1}(\mpplus_{j \neq n} A_{ij} \vect{x}_j\oplus \vect{c}_i) \quad \text{if} \ I\neq \emptyset \; ,\\
0 \quad \text{otherwise}.
\end{cases}
\]
Obviously, $\lambda' \sleq \lambda$. Besides, for all $i\in I$ we have 
\[
(\mpplus_{j \neq n} A_{ij} \vect{x}_j ) \oplus A_{in} \lambda' \oplus \vect{c}_i \sleq ( \mpplus_{j \neq n} B_{ij} \vect{x}_j ) \oplus B_{in} \lambda' \oplus \vect{d}_i \; ,
\]
since $\mpplus_{j \neq n} A_{ij} \vect{x}_j\oplus \vect{c}_i < \abs{B_{in}} \lambda'$ and $A_{in}=\mpzero$. 

Now assume $\lambda \in \pert$. Let $I'$ be the set of indices $i \in I$ such that 
$B_{in} \lambda \sgreater \mpplus_{j \neq n} B_{ij} \vect{x}_j \mpplus \vect{d}_i$. 
Then,  
\[
(\mpplus_{j} A_{ij} \vect{y}_j ) \mpplus \vect{c}_i  = (\mpplus_{j \neq n} A_{ij} \vect{x}_j ) \mpplus \vect{c}_i \sleq B_{in} \lambda 
\]
for all $i \in I'$. For $i\in I'$, let $\nu_i=\mpplus_{j \neq n} A_{ij} \vect{x}_j \mpplus \vect{c}_i$. 
Since $\nu_i$ belongs to $\maxplus$ ($A_{ij}$, $\vect{x}_j$ and $\vect{c}_i$ belong to $\maxplus$) 
and $B_{in} \lambda $ belongs to $\pert$, we necessarily have $\nu_i < \abs{B_{in} \lambda}$. 
Define $\lambda' := \kappa \abs{\lambda} \in \R$, where
\[
\kappa  =
\begin{cases}
\max_{i \in I'} (\nu_i - \abs{B_{in} \lambda})/2 \quad \text{if} \ \{i\in I'\mid \nu_i\in \R \} \neq \emptyset \; ,\\
-1 \quad \text{otherwise}.
\end{cases}
\]
As $\kappa < 0$, we have $\lambda' \sleq \lambda$.
Moreover, if $i \in I'$, then $\nu_i < \kappa \abs{B_{in} \lambda}$ and so 
\[
\begin{split}
(\mpplus_{j} A_{ij} \vect{y}'_j ) \mpplus \vect{c}_i & \sleq ( \mpplus_{j} A_{ij} \vect{y}_j ) \mpplus \vect{c}_i = ( \mpplus_{j \neq n} A_{ij} \vect{x}_j ) \mpplus \vect{c}_i =  \\ 
& = \nu_i\sleq \kappa B_{in} \lambda \sleq \kappa B_{in} \abs{\lambda} = 
B_{in} \lambda' \sleq ( \mpplus_{j} B_{ij} \vect{y}'_j ) \mpplus \vect{d}_i \enspace.
\end{split}
\]
Finally, note that for $i \in I \setminus I'$ the relations in~\eqref{eq:common_case} are still valid because in that case we have $\mpplus_{j \neq n} B_{ij} \vect{x}_j \mpplus \vect{d}_i \sgeq B_{in} \lambda \sgeq B_{in} \lambda'$. 
This completes the proof. 
\end{proof}

\subsection{Characterization of tropical polyhedra with mixed constraints in terms of zones and tropical hemispaces}

In this subsection we discuss some consequences of tropical Fourier-Motzkin elimination. In the first place, we establish that the tropical convex hull of the union of two tropical polyhedra with mixed constraints is a tropical polyhedron with mixed constraints. Recall that the \emph{tropical convex hull} $\tconv(G)$ of a set $G \subset \maxplus^n$ is defined as the set of the vectors of the form 
\[
\lambda_1 \vect{x}^1 \mpplus \cdots \mpplus \lambda_m \vect{x}^m
\]
where $m$ is a positive integer, $\vect{x}^i \in G$, $\lambda_i \in \maxplus$ ($i \in \oneto{m})$, and $\mpplus_{i \in \oneto{m}} \lambda_i = 0$. 

\begin{proposition}\label{prop:over_approx}
Let $\PP, \PP'\subset \maxplus^n$ be polyhedra with mixed constraints, respectively defined by the systems 
$A \vect{x} \mpplus \vect{c} \sleq B \vect{x} \mpplus \vect{d}$ and $A'\vect{x} \mpplus \vect{c}' \sleq B' \vect{x} \mpplus \vect{d}'$, 
and let $\QQ$ be the polyhedron defined by the inequalities which are obtained eliminating $\vect{y}_1,\ldots ,\vect{y}_n,\vect{y}'_1,\ldots ,\vect{y}'_n,\lambda$ and $\mu$ in the following system:
\begin{align}
& \vect{x}  = \vect{y} \mpplus \vect{y}' \;  && \lambda \mpplus \mu  = \mpone \;  \notag \\
& A \vect{y} \mpplus \lambda \vect{c}  \sleq B \vect{y} \mpplus \lambda \vect{d}  \;  && \vect{y}_1\mpplus\cdots \mpplus \vect{y}_n \sleq (+\infty) \lambda \;  \label{eq:over_approx} \\ 
& A' \vect{y}' \mpplus \mu \vect{c}' \sleq B' \vect{y}' \mpplus \mu \vect{d}'  \;  && \vect{y}'_1\mpplus\cdots \mpplus \vect{y}'_n \sleq (+\infty) \mu \;  \notag
\end{align}
Then, we have $\tconv(\PP \cup \PP') = \QQ$.
\end{proposition}

\begin{proof}
In the first place, note that $\PP \subset \QQ$. Indeed, given $\vect{x} \in \PP$, let $\lambda = 0$, $\mu = -\infty$, $\vect{y} = \vect{x}$ and $\vect{y}'_i = -\infty$ for $i\in \oneto{n}$. Then, \eqref{eq:over_approx} is clearly satisfied, and so from Theorem~\ref{th:fourier_motzkin_maxplus_case} we deduce that $\vect{x} \in \QQ$. Similarly, $\PP' \subset \QQ$. Since $\QQ$ is tropically convex, we deduce that $\tconv(\PP \cup \PP') \subset \QQ$.

Conversely, let $\vect{x} \in \QQ$, and consider $\vect{y}, \vect{y}', \lambda, \mu$ as in~\eqref{eq:over_approx}. If both $\lambda$, $\mu$ are distinct from $-\infty$, then $\lambda^{-1} \vect{y} \in \PP$, $\mu^{-1} \vect{y}' \in \PP'$, and $\vect{x} = \lambda (\lambda^{-1} \vect{y}) \mpplus \mu (\mu^{-1} \vect{y}')$, which ensures that $\vect{x} \in \tconv(\PP \cup \PP')$.
Otherwise, if for instance $\mu=\mpzero$, then we necessarily have $\lambda = 0$, and so $\vect{y}\in \PP$. Moreover, from $\vect{y}'_1\mpplus\cdots \mpplus \vect{y}'_n \sleq (+\infty) \mu$ we deduce that $\vect{y}'_i=\mpzero$ for all $i\in \oneto{n}$, and in consequence $\vect{x}= \vect{y}$. This completes the proof, because again we have $\vect{x}= \vect{y} \in \PP \subset \tconv(\PP \cup \PP')$, and so $\QQ \subset \tconv(\PP \cup \PP')$.
\end{proof}

In order to characterize tropical polyhedra with mixed constraints in terms of zones, we first need to extend the definition of zones to $\maxplus$: 
a zone (of $\maxplus^n$) is a set of vectors $\vect{x} \in \maxplus^n$ defined by inequalities of the form 
\begin{equation} \label{eq:proof_aux}
m_i \lhd \vect{x}_i \qquad \vect{x}_i \lhd M_i \qquad  \vect{x}_i \lhd k_{ij} \vect{x}_j  
\end{equation}
where $\lhd \in \{ \leq , < \}$ and $m_i, M_i, k_{ij} \in \maxplus$. 

\begin{theorem}\label{th:union_of_zones}
A subset $\PP$ of $\maxplus^n$ is a tropical polyhedron with mixed constraints if, and only if, it is a tropically convex union of finitely many zones. 
\end{theorem}

\begin{proof}
Observe that when $\lhd$ is equal to $\leq$ in~\eqref{eq:proof_aux}, these inequalities are equivalent to the ones obtained replacing $\lhd$ by $\sleq$. 
Moreover, we have
\[
\begin{array}{r@{{}<{}}l@{{}\iff{}}l}
m_i & \vect{x}_i &
\begin{cases}
m_i \sleq \myul{0} \vect{x}_i & \text{if} \   m_i \in \R \\ 
0 \sleq (+\infty)  \vect{x}_i & \text{if} \  m_i =\mpzero 
\end{cases} \\[3ex]
\vect{x}_i & M_i & \vect{x}_i \sleq \myul{0} M_i \quad \text{if} \  M_i \in \R \\[2ex]
\vect{x}_i & k_{ij} \vect{x}_j & ( \vect{x}_i \sleq (\myul{0}) k_{ij}  \vect{x}_j \; \text{and} \; 0\sleq (+\infty) \vect{x}_j) \quad  \text{if} \  k_{ij} \in \R 
\end{array}
\]
In consequence, any zone is a tropical polyhedron with mixed constraints. Then, by Proposition~\ref{prop:over_approx}, so is the tropical convex hull of the union of finitely many zones. The ``if'' part  of the theorem follows from the fact that if the union of finitely many zones is tropically convex, then it coincides with its tropical convex hull. 

Suppose now that $\QQ$ is a polyhedron defined by mixed inequalities of the form~\eqref{eq:mixed_affine_inequality}, in which only one of the coefficients $b_j$ is distinct from $-\infty$. 
If this coefficient is $b_0$, as we can assume that $b_0\neq +\infty$, it follows that~\eqref{eq:mixed_affine_inequality} can be rewritten as a system of inequalities of the form  $\vect{x}_i \sleq \beta M_i$, with $\beta \in \{0, \myul{0}\}$ and $M_i\in \R$. If the considered coefficient is $b_j$ for $j \in \oneto{n}$, then~\eqref{eq:mixed_affine_inequality} can be rewritten as a system of inequalities of the form $m_j \sleq \beta \vect{x}_j$ and $\vect{x}_i \sleq \beta k_{ij} \vect{x}_j$, with $\beta \in \{0, \myul{0},+\infty \}$ and $m_j, k_{ij} \in \R$. Since 
\begin{align*}
& \vect{x}_i \sleq \beta k_{ij} \vect{x}_j \iff 
\begin{cases}
 \vect{x}_i \leq  k_{ij} \vect{x}_j & \text{if} \  \beta= 0 \\
 \vect{x}_i < k_{ij} \vect{x}_j \; \text{or} \; (\vect{x}_i \leq \mpzero \; \text{and} \;\vect{x}_j \leq \mpzero ) & \text{if} \  \beta= \myul{0} \\
 \mpzero <  \vect{x}_j \; \text{or} \; (\vect{x}_i \leq \mpzero \; \text{and} \;\vect{x}_j \leq \mpzero ) & \text{if} \  \beta= +\infty 
\end{cases} \\
& \vect{x}_i \sleq \beta M_i \iff 
\begin{cases}
\vect{x}_i \leq  M_i & \text{if} \  \beta= 0 \\ 
\vect{x}_i <  M_i & \text{if} \  \beta= \myul{0} 
\end{cases} 
\quad \text{and} \quad
 m_j \sleq \beta \vect{x}_j \iff
\begin{cases}
m_j \leq  \vect{x}_j & \text{if} \   \beta=0 \\ 
m_j < \vect{x}_j & \text{if} \   \beta = \myul{0} \\ 
-\infty < \vect{x}_j & \text{if} \  \beta = +\infty 
\end{cases} 
\end{align*}
it follows that $\QQ$ is a finite union of zones. 

The ``only if'' part of the theorem follows from the fact that any polyhedron with mixed constraints can be written as a finite union of polyhedra of the form considered in the previous paragraph (it suffices to choose, in each inequality, one term in the right-hand side to be the maximizing one). 
\end{proof}

Theorem~\ref{th:union_of_zones} raises the problem, of independent
interest, of determining whether a given union of finitely many zones is
tropically convex. To our knowledge, this problem has not been studied
so far, even in the case of closed zones. A naive solution consists in
computing the tropical convex hull of the union of zones (using
Proposition~\ref{prop:over_approx}), and checking whether it intersects
the complement of the union of zones (the latter can be expanded into a
union of zones, and the intersection test requires the techniques developed in Section~\ref{subsec:elimination_step}, see Theorem~\ref{th:equivalence_MPG}). This approach would be particularly expensive. Yet, it is similar to the technique implemented in the UPPAAL DBM library~\cite{DBMLib} to test if a union of zones is a zone. Whether there exists a more efficient method is left for future work.

Now we study the relationship between tropical polyhedra with mixed constraints and tropical hemispaces. Recall that a \emph{tropical hemispace} is a tropically convex set whose complement is also tropically convex. Tropical hemispaces and mixed half-spaces share the property that their closure is a closed half-space
(in the case of hemispaces, this is proved in~\cite{BH-08}). In fact, Proposition~\ref{prop:hemispace} and Example~\ref{remark:hemispace} below show that mixed half-spaces are a proper subclass of hemispaces.
\begin{proposition}\label{prop:hemispace}
Mixed half-spaces are tropical hemispaces. 
\end{proposition}

\begin{proof}
We already proved in Proposition~\ref{prop:mixed_polyhedra_convex} that mixed half-spaces are tropically convex.

The complement of the mixed half-space defined by~\eqref{eq:mixed_affine_inequality} consists of the vectors $\vect{x} \in \maxplus^n$ such that 
\begin{equation}
a_0 \mpplus a_1 \vect{x}_1 \mpplus \cdots \mpplus a_n \vect{x}_n \sgreater b_0 \mpplus b_1 \vect{x}_1 \mpplus \cdots \mpplus b_n \vect{x}_n \ . \label{eq:complement}
\end{equation}
Let $\vect{x}, \vect{y} \in \maxplus^n$ be in this complement, and $\lambda, \mu \in \maxplus$ be such that $\lambda \mpplus \mu = 0$. Without loss of generality, assume that $\lambda = 0$. If $\mu = -\infty$, then $\lambda \vect{x} \mpplus \mu \vect{y} = \vect{x}$ trivially satisfies~\eqref{eq:complement}. Otherwise, \ie\ if $\mu \in \R$ and $\mu \leq 0$, then by Property~\eqref{lemma:smaxplus:P4} of Lemma~\ref{lemma:smaxplus} we have $\mu (a_0 \mpplus a_1 \vect{y}_1 \mpplus \cdots \mpplus a_n \vect{y}_n) \sgreater \mu (b_0 \mpplus b_1 \vect{y}_1 \mpplus \cdots \mpplus b_n \vect{y}_n)$. It readily follows that 
$a_0 \mpplus a_1 ( \lambda \vect{x}_1 \mpplus \mu \vect{y}_1 ) \mpplus \cdots \mpplus a_n ( \lambda \vect{x}_n \mpplus \mu \vect{y}_n ) \sgreater b_0 \mpplus b_1 ( \lambda \vect{x}_1 \mpplus \mu \vect{y}_1 ) \mpplus \cdots \mpplus b_n ( \lambda \vect{x}_n \mpplus \mu \vect{y}_n )$. 
This shows that the complement of a mixed half-space is tropically convex.
\end{proof}

\begin{example}\label{remark:hemispace}
Consider the tropical hemispace $\HH$ defined as the set of vectors $\vect{x} \in \maxplus^4$ such that
\begin{align*}
 (\vect{x}_3 \leq \vect{x}_1 \ \text{and} \ \vect{x}_4 \leq \vect{x}_1) \quad \text{or} \quad  (\vect{x}_3 \leq \vect{x}_2 \ \text{and} \ \vect{x}_4 < \vect{x}_2) \; .
\end{align*}
We claim that $\HH$ is not a mixed half-space. Its closure is the tropical half-space defined by the inequality 
\begin{equation}
\vect{x}_3 \mpplus \vect{x}_4 \leq 0 \vect{x}_1 \mpplus 0 \vect{x}_2 \ . \label{eq:closed_hs}
\end{equation}
Any mixed half-space whose closure is given by~\eqref{eq:closed_hs} is defined by a mixed inequality obtained from~\eqref{eq:closed_hs} by replacing some of the coefficients $0$ on the right-hand side by $\myul{0}$. However, it can be easily verified that none of these mixed half-spaces is equal to $\HH$. 

Similarly, it can be checked that the complement of $\HH$, which is given by the set of vectors $\vect{x} \in \maxplus^4$ such that
\[
\vect{x}_1 < \vect{x}_3 \mpplus \vect{x}_4 \quad \text{and} \quad (\vect{x}_2 < \vect{x}_3 \ \text{or}\ \vect{x}_2 \leq \vect{x}_4) \ ,
\]
is not a mixed half-space either.
\end{example}

To prove the following proposition, we shall use the characterization of tropical hemispaces in terms of $(P,R)$-decompositions established in~\cite{KNS}. 

\begin{proposition}\label{prop:hemispace0}
Tropical hemispaces are tropical polyhedra with mixed constraints.
\end{proposition}

\begin{proof}
In the first place, let us recall that the \emph{tropical conic hull} of a subset $G$ of $\maxplus^n$ is defined as
\[
\tcone(G) \mydef \{ \mu_1 \vect{x}^1 \mpplus \cdots \mpplus \mu_m \vect{x}^m \mid m \in \mathbb{N}, \ \vect{x}^i \in G, \ \mu_i \in \maxplus \} \; .
\]
Moreover, given two subsets $G$ and $G'$ of $\maxplus^n$, their tropical Minkowski sum $G \mpplus G'$ is defined as $\{ \vect{x} \mpplus \vect{x}' \mid \vect{x} \in G, \ \vect{x}' \in G' \}$. 
In this proof, we denote by $\vect{u}^i \in \maxplus^n$, for $i\in \oneto{n}$, the vector defined by $\vect{u}^i_j \mydef 0$ if $j = i$ and $\vect{u}^i_j \mydef -\infty$ otherwise.

Let $\HH_1, \HH_2$ be two complementary hemispaces. Suppose that the vector all of whose entries are equal to $-\infty$ belongs to $\HH_1$. Then, by~\cite[Theorem~4.22]{KNS} there exist a partition $I, J$ of $\oneto{n}$ and interval sets $\Lambda^1_{ij}, \Lambda^2_{ij} \subset \cmaxplus$ for $i\in I \cup \{0\}$ and $j \in J$ such that
\begin{align*}
\HH_1 & = \tconv(\{\lambda \vect{u}^j \mid j \in J, \lambda \in \Lambda_{0j}^1\}) \mpplus \tcone(\{\vect{u}^i \mpplus \lambda \vect{u}^j \mid i \in I, \ j \in J, \ \lambda \in \Lambda_{ij}^1 \})  \notag \\
\HH_2 & = \tconv(\{\lambda \vect{u}^j \mid j \in J, \lambda \neq +\infty, \ \lambda \in \Lambda_{0j}^2\}) \mpplus \tcone(\{ \lambda \vect{u}^i \mpplus \vect{u}^j \mid i \in I, \ j \in J, \ \lambda \in -\Lambda_{ij}^2 \}) 
\end{align*}
where each couple of intervals $\Lambda^1_{ij}, \Lambda^2_{ij}$ has one of the following forms:
\begin{equation}
(\Lambda^1_{ij}, \Lambda^2_{ij}) = 
\begin{cases}
([-\infty, \beta], ]\beta, +\infty]), & \beta \in \R \cup \{-\infty\} \\
([-\infty, \alpha[, [\alpha, +\infty]), & \alpha \in \R \cup \{+\infty\}
\end{cases} \label{eq:proof_hemispace2}
\end{equation}
Thus, note that $\vect{x} \in \HH_1$ if, and only if, there exist $\lambda_{ij}$ and $\mu_{ij}$ in $\maxplus$ for $i\in I \cup \{0\}$ and $j \in J$ such that 
\begin{align*}
\vect{x}_i & = \mpplus_{j \in J} \mu_{ij} && \text{for} \; i \in I \\
\vect{x}_j & = \mpplus_{i \in I \cup \{0\}} \lambda_{ij} \mu_{ij} && \text{for} \; j \in J \\
\mpplus_{j \in J} \mu_{0j} & = 0 \\
\lambda_{ij}&  \in \Lambda^1_{ij} && \text{for} \; i \in I \cup \{0\} \; \text{and} \; j \in J 
\end{align*}
Equivalently, $\vect{x} \in \HH_1$ if, and only if, there exist $\mu_{ij}$ and $\nu_{ij}$ in $\maxplus$ for $i\in I \cup \{0\}$ and $j \in J$ such that
\begin{equation}
\begin{aligned}
\vect{x}_i & = \mpplus_{j \in J} \mu_{ij} && \text{for} \; i \in I \\
\vect{x}_j & = \mpplus_{i \in I \cup \{0\}} \nu_{ij} && \text{for} \; j \in J \\
\mpplus_{j \in J} \mu_{0j} & = 0 \\
\nu_{ij} &  \sleq 
\begin{cases}
\beta \mu_{ij} & \text{if} \ \Lambda^1_{ij} = [-\infty,\beta] \ \text{with} \ \beta \in \R \cup \{-\infty\}\\
\myul{\alpha} \mu_{ij} & \text{if} \ \Lambda^1_{ij} = [-\infty, \alpha[ \ \text{with} \ \alpha \in \R \\
(+\infty) \mu_{ij} & \text{if} \ \Lambda^1_{ij} = [-\infty, \alpha[ \ \text{with} \ \alpha = +\infty  
\end{cases} && \text{for} \; i \in I \cup \{0\} \; \text{and} \;  j \in J
\end{aligned}\label{eq:proof_aux3}
\end{equation}
By Theorem~\ref{th:fourier_motzkin_maxplus_case}, we conclude that $\HH_1$ is a tropical polyhedron with mixed constraints, since it is defined by the system of mixed inequalities obtained by eliminating $\mu_{ij}$ and $\nu_{ij}$  for $i \in I \cup \{0\}$ and $j \in J$ in~\eqref{eq:proof_aux3}. 

The same result can be obtained for $\HH_2$ using a symmetric argument.
\end{proof}

The following result is a straightforward consequence of Propositions~\ref{prop:hemispace} and~\ref{prop:hemispace0}.
\begin{corollary}\label{cor:intersection_of_hemispaces}
Tropical polyhedra with mixed constraints are precisely the intersections of finitely many tropical hemispaces. 
\end{corollary}

\begin{remark}
We have previously defined a closed tropical polyhedron as the solution set of inequalities of the form~\eqref{eq:affine_inequality}, \ie \ mixed inequalities with coefficients in $\maxplus$. We point out that this definition is consistent with the fact that such polyhedra are precisely the tropical polyhedra with mixed constraints which are closed. Indeed, by Theorem~\ref{th:union_of_zones} a tropical polyhedron with mixed constraints $\PP$ can be written as a finite union of zones. If $\PP$ is closed, then it is equal to the union of the closure of these zones (\ie\ closed zones, which can be generated by finitely many points and rays by the tropical Minkowski-Weyl Theorem, see~\cite[Theorem~2]{GK09}). Since $\PP$ is tropically convex, it is even equal to the tropical convex hull of this union. It follows that it is generated by finitely many points and rays. By the tropical Minkowski-Weyl Theorem, we deduce that it is the solution set of finitely many inequalities of the form~\eqref{eq:affine_inequality}.
\end{remark}

\section{Eliminating redundant mixed inequalities}\label{subsec:elimination_step}

Like in the classical setting, 
Fourier-Motzkin
elimination generates a system of $O(p^2)$ inequalities from an input
with $p$ constraints. Consequently, the number of inequalities may grow 
in $O(p^{2^k})$ after $k$ successive applications. 
To avoid the explosion of the size of the constraint system, superfluous
inequalities must be eliminated. With this aim, 
we present a decision procedure for implications of the form:
\begin{equation}\label{eq:implication}
A \vect{x} \mpplus \vect{c} \sleq B \vect{x} \mpplus \vect{d} \implies \vect{e} \vect{x} \mpplus g \sleq \vect{f} \vect{x} \mpplus h 
\quad \ \text{for all} \ \vect{x}\in \maxplus^n \; ,
\end{equation}
where $A \in \maxplus^{p \times n}$, $B \in \smaxplus^{p \times n}$, 
$\vect{c} \in \maxplus^p$, $\vect{d} \in \smaxplus^p$, $g \in \maxplus$, $h \in \smaxplus$, and $\vect{e}$ and $\vect{f}$ are $n$-dimensional row vectors with entries in $\maxplus$ and $\smaxplus$ respectively. 
We assume that $\vect{e}$ and $g$ are not identically null (in the tropical sense) and that $h\neq +\infty$, because otherwise deciding implication~\eqref{eq:implication} is trivial. 

\subsection{Equivalence with mean payoff games}\label{Subsection:EquivMPG}

We first show that deciding an implication of the form~\eqref{eq:implication} is equivalent to an emptiness problem for tropical polyhedra with mixed constraints. 

\begin{proposition}\label{prop:EquivImplEmpt}
Let $\QQ$ be the tropical polyhedron with mixed constraints 
defined by the system $A \vect{x} \mpplus \vect{c} \sleq B \vect{x} \mpplus \vect{d}$ and the following inequalities:
\begin{equation}\label{P1ofprop:EquivImplEmpt}
\begin{cases}
\vect{f}_i \vect{x}_i \sleq (\myul{0} \vect{e}) \vect{x} \mpplus \myul{0} g & \text{if} \ \vect{f}_i \in \maxplus \\
\abs{\vect{f}_i} \vect{x}_i \sleq \vect{e} \vect{x} \mpplus g & \text{if} \ \vect{f}_i \in \pert \\
\vect{x}_i \sleq \mpzero & \text{if} \ \vect{f}_i = +\infty
\end{cases}
\end{equation}
for all $i \in \oneto{n}$,
\begin{equation}\label{P2ofprop:EquivImplEmpt}
\begin{cases}
h \sleq (\myul{0} \vect{e}) \vect{x} \mpplus \myul{0} g  &\text{if} \ h\in \maxplus \\
\abs{h} \sleq \vect{e} \vect{x} \mpplus g  &\text{if} \ h\in \pert 
\end{cases}
\end{equation}
and 
\[
0 \sleq \mpplus_{\vect{e}_i \neq -\infty} (+\infty) \vect{x}_i \quad \text{if} \ g = -\infty .  
\]
Then, implication~\eqref{eq:implication} holds if, and only if, $\QQ$ is empty.
\end{proposition}

\begin{proof}
Implication~\eqref{eq:implication} is false if, and only if, there exists $\vect{x} \in \maxplus^n$ such that $A \vect{x} \mpplus \vect{c} \sleq B \vect{x} \mpplus \vect{d}$ and $\vect{e} \vect{x} \mpplus g \sgreater \vect{f} \vect{x} \mpplus h $. 
Observe that $\vect{e} \vect{x} \mpplus g \sgreater \vect{f} \vect{x} \mpplus h $ holds if, and only if, $\vect{e} \vect{x} \mpplus g \sgreater h $ and $\vect{e} \vect{x} \mpplus g \sgreater \vect{f}_i \vect{x}_i$ for all $i \in \oneto{n}$. 
This implies $\vect{e} \vect{x} \mpplus g > \mpzero$, \ie \  $0 \sleq (+\infty) (\vect{e} \vect{x} \mpplus g)$. The latter inequality is trivially satisfied if $g \neq -\infty$. When $g = -\infty$, it is equivalent to $0 \sleq \mpplus_{\vect{e}_i \neq -\infty} (+\infty) \vect{x}_i$. 
Finally, assuming $\vect{e} \vect{x} \mpplus g > \mpzero$, note that $\vect{e} \vect{x} \mpplus g \sgreater \vect{f}_i \vect{x}_i$ 
is equivalent to~\eqref{P1ofprop:EquivImplEmpt}, and $\vect{e} \vect{x} \mpplus g \sgreater h $ is equivalent to~\eqref{P2ofprop:EquivImplEmpt}.  
This completes the proof.
\end{proof}

In light of Proposition~\ref{prop:EquivImplEmpt}, it is enough to develop a decision procedure to determine whether a
tropical polyhedron with mixed constraints is empty.
Our approach relies on \emph{parametric mean payoff games}, following
the lines of~\cite{AkianGaubertGutermanIJAC2011,AllamigeonGaubertKatzLAA2011}. 

Let $\mcR$ be a tropical polyhedron with mixed constraints defined by the system 
\begin{equation}
M \vect{x} \mpplus \vect{p} \sleq N \vect{x} \mpplus \vect{q} \;, \label{eq:r_system}
\end{equation}
where $M \in \maxplus^{r \times n}$, $N \in \smaxplus^{r \times n}$,
$\vect{p} \in \maxplus^r$ and $\vect{q} \in \smaxplus^r$. 
It is convenient to consider two different cases, 
depending on whether $+\infty$ coefficients appear or not in~\eqref{eq:r_system}. 

\subsubsection{Polyhedra defined by systems with no $+\infty$ coefficients.}

In the first place, we restrict to the following case: 

\begin{assumption}\label{ass:non_infty}
No coefficient is equal to $+\infty$ on the right-hand side of mixed inequalities.
\end{assumption}
In other words, in this subsection we assume $N \in (\smaxplus \setminus \{+\infty\})^{r \times n}$ and $\vect{q} \in (\smaxplus \setminus \{+\infty\})^r$. 

In this case, with each $\eps \geq 0$ we associate a closed tropical polyhedron $\mcR_\eps$, given by the system
$M \vect{x} \mpplus \vect{p} \sleq N(\eps) \vect{x} \mpplus
\vect{q}(\eps)$, and a mean payoff game involving two players, ``Max''
and ``Min'', playing on a weighted bipartite digraph $\GG_\eps$ composed
of two kinds of nodes: \emph{row nodes} $i \in [r]$, and
\emph{column nodes} $j \in [n+1]$. 
This digraph contains an arc
from row node $i$ to column node $j$ with weight $N_{ij}(\eps)$ when
$N_{ij} \neq \mpzero$, and an arc from $j$ to $i$ with weight
$-M_{ij}$ when $M_{ij} \neq \mpzero$.  Similarly, it contains an arc
from row node $i$ to column node $n+1$ with weight $\vect{q}_i(\eps)$ if
$\vect{q}_i \neq \mpzero$, and an arc from column node $n+1$ to
row node $i$ with weight $-\vect{p}_i$ when $\vect{p}_i \neq \mpzero$.
\begin{example} 
Figure~\ref{fig:mean_payoff_game} provides an illustration of the
digraph $\GG_\eps$ corresponding to the tropical polyhedron with mixed constraints defined by the system:
\[
\begin{aligned}
1: && -3 & \sleq \vect{x}_1 \\
2: && 0 & \sleq 1\vect{x}_1 \mpplus \myul{0}\vect{x}_2 
\end{aligned}
\quad \quad
\begin{aligned}
3: && -2 & \sleq \vect{x}_2 \\
4: && (-2)\vect{x}_1 & \sleq \myul{0} \mpplus (-1) \vect{x}_2 
\end{aligned}
\]

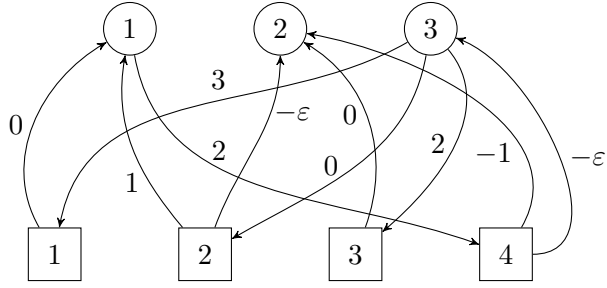
\begin{figure}[tbp]
\begin{center}
\begin{tikzpicture}[scale=2,rownode/.style={rectangle,draw=black,minimum size=0.7cm},colnode/.style={circle,draw=black,minimum size=0.7cm},>=stealth']

\node[rownode] (row1) at (0,0) {$1$};
\node[rownode] (row2) at (1,0) {$2$};
\node[rownode] (row3) at (2,0) {$3$};
\node[rownode] (row4) at (3,0) {$4$};

\node[colnode] (col1) at (0.5,1.5) {$1$};
\node[colnode] (col2) at (1.5,1.5) {$2$};
\node[colnode] (col3) at (2.5,1.5) {$3$};

\draw (row1) edge[->,out=120,in=210] node[left] {$0$} (col1) ;
\draw (col3) edge[->,out=210,in=80] node[above right=0.1cm] {$3$} (row1);

\draw (row2) edge[->,out=130,in=-100] node[below=0.22cm] {$1$} (col1) ;
\draw (row2) edge[->,out=70,in=-90] node[above=0.4cm,right=0.1cm] {$-\eps$} (col2) ;
\draw (col3) edge[->,out=-100,in=30] node[below,left=0.1cm] {$0$} (row2);

\draw (row3) edge[->,out=70,in=-30] node[above=0.1cm,left=0.02cm] {$0$} (col2) ;
\draw (col3) edge[->,out=-50,in=40] node[left] {$2$} (row3);

\draw (row4) edge[->,out=0,in=-20] node[right] {$-\eps$} (col3) ;
\draw (row4) edge[->,out=60,in=-10] node[below=0.7cm,left=-0.6cm] {$-1$} (col2) ;
\draw (col1) edge[->,out=-80,in=160] node[above=0.48cm,left=0.35cm] {$2$} (row4);
\end{tikzpicture}
\end{center}
\caption{The digraph associated with a parametric mean payoff game (column and row nodes are respectively represented by circles and squares).}\label{fig:mean_payoff_game}
\end{figure}
\end{example}

The principle of the game is the following. Players Min and Max
alternatively move a pawn over the nodes of $\GG_\eps$. When it is
placed on a column node, Player Min selects an outgoing arc, moves the pawn to
the corresponding row node, and pays to Player Max the weight of the selected arc. Once
the pawn is on a row node, Player Max similarly selects an outgoing arc, moves
the pawn along it, and receives from Player Min a payment equal to the weight
of the selected arc.

In the sequel, we suppose that Players Max and Min always have at least one possible action in each node: 
\begin{assumption}\label{AssumpGame}
Each node of $\GG_\eps$ has at least one outgoing arc. 
\end{assumption}
This technical property can be assumed without loss of generality, up to adding trivial inequalities or removing non-relevant unknowns in the system. We refer to the discussion of Assumptions~2.1 and 2.2 in~\cite{AkianGaubertGutermanIJAC2011} for further details.
 
We consider infinite runs of the game, in which case the payoff is
defined as the mean of the payments of Player Min to Player Max. Player Min wants to minimize this mean of payments while Player Max wants to maximize it. We denote by
$v(\eps)$ the value of the game associated with $\GG_\eps$ when it
starts at column node $n+1$. It is shown
in~\cite[Theorem~3.5]{AkianGaubertGutermanIJAC2011} that the tropical
polyhedron $\mcR_\eps$ is non-empty if, and only
if, 
$v(\eps) \geq 0$, \ie{} column node $n+1$ is a winning initial node (for
Player Max). Then, the following result immediately follows from Lemma~\ref{lemma:val}.

\begin{proposition}\label{prop:emptiness_chi}
The tropical polyhedron with mixed constraints $\mcR$ is non-empty if, and only if, there exists $\eps > 0$ such that $v(\eps) \geq 0$.
\end{proposition}

Let $\tilde{M}$ (resp.\ $\tilde{N}$) be the matrix of size $r \times (n+1)$ obtained by concatenating matrix $M$ and column vector $\vect{p}$ (resp.\ $N$ and $\vect{q}$). The dynamic programming operator $g_\eps$ associated with the game over $\GG_\eps$ is the function from $\maxplus^{n+1}$ to itself defined by
\begin{equation}
(g_\eps(\vect{x}))_j \defi 
\min_{\substack{i \in [r]\\ \tilde{M}_{ij} \neq \mpzero}}
\Bigl( -\tilde{M}_{ij} + \max_{\substack{k \in [n+1]\\ \tilde{N}_{ik} \neq \mpzero}} \Bigl( \tilde{N}_{ik}(\eps) + \vect{x}_k\Bigr) \Bigr)
\label{eq:operator} \enspace ,
\end{equation}
for $j \in [n+1]$. 
This function satisfies the following properties:
\begin{enumerate}[(i)]
\item it is order preserving, \ie{} $\vect{x} \leq \vect{y}$ implies $g_\eps(\vect{x}) \leq g_\eps(\vect{y})$ for all $\vect{x}, \vect{y} \in \maxplus^{n+1}$,
\item it is additively homogeneous, \ie{} $g_\eps(\lambda \vect{x}) =
\lambda  g_\eps(\vect{x})$ for all $\lambda \in \maxplus$ and $\vect{x} \in \maxplus^{n+1}$,
\item it preserves $\Real^{n+1}$, thanks to Assumption~\ref{AssumpGame} above.
\end{enumerate}
Such a function can be shown to be non-expansive for the sup-norm. Since it is also piecewise affine, a theorem due to Kohlberg~\cite{Kohlberg1980} implies the following vector $\chi(g_\eps)$, referred to as the \emph{cycle-time vector} of $g_\eps$, exists and has finite entries:
\[
\chi(g_\eps) \defi \lim_{h \rightarrow +\infty} g^h_\eps(0)/h \enspace.
\]
Kohlberg's theorem also implies the $j$-th entry of $\chi(g_\eps)$, which we denote by $\chi_j(g_\eps)$, corresponds to the value of the game when it starts at column node $j$.
We refer to~\cite{AkianGaubertGutermanIJAC2011} for further details. Following the notation above, we consequently have
\[
v(\eps) = \chi_{n+1}(g_\eps) 
\]
for $\eps \geq 0$.

The cycle-time vector $\chi(g_\eps)$ can be expressed in terms of the
cycle-time vectors of dynamic programming operators associated with
certain one-player games.  More precisely, a \emph{(positional)
  strategy} for Player Min is a function $\sigma: [n+1] \to [r]$ associating with each column node $j$ a row node
$\sigma(j)$ such that $\tilde{M}_{\sigma(j)j} \neq \mpzero$. Such a
strategy defines a one-player game (played by Player Max) over the
digraph $\GG_\eps^\sigma$ obtained from $\GG_\eps$ by deleting the arcs
connecting column nodes $j$ with row nodes $i$ such that $i \neq
\sigma(j)$. Its dynamic programming operator $g_\eps^\sigma$ is given by:
\[
(g_\eps^\sigma(\vect{x}))_j = -\tilde{M}_{\sigma(j)j} +
\max_{\substack{k\in \oneto{n+1}\\ \tilde{N}_{\sigma(j)k} \neq \mpzero}}
\left( \tilde{N}_{\sigma(j)k}(\eps) + \vect{x}_k \right) 
\enspace,
\]
for $j \in \oneto{n+1}$.  Observe that this operator is linear in
the tropical semiring $\maxplus$. If we denote by $\Sigma$ the (finite)
set of strategies for Player Min, then as another consequence of
Kohlberg's theorem, it can be shown that
\begin{equation}\label{ChiMinSigma}
\chi(g_\eps) = \min_{\sigma \in \Sigma}
\chi(g_\eps^\sigma) \; .  
\end{equation}
A dual result based on strategies for Player Max
can also be established.

Given $\sigma \in \Sigma$, the $(n+1)$-th entry of the vector $\chi(g_\eps^\sigma)$ can be similarly interpreted as the value of the one-player game associated with the digraph $\GG_\eps^\sigma$ when it starts at column node $n+1$. As the function $g_\eps^\sigma$ is linear in the tropical semiring, it is known~\cite{CochetGaubertGunawardena99} that $\chi_{n+1}(g_\eps^\sigma)$ is equal to the maximal weight-to-length ratio of the elementary circuits of $\GG_\eps^\sigma$ reachable from column node $n+1$. A circuit in this digraph is referred to as a sequence of column nodes $j_0, \dots, j_{l-1}, j_l = j_0$, where $l \geq 1$, and so $l$ is considered to be its length. Note that the reachability relation in $\GG_\eps^\sigma$ does not depend on $\eps$. Let $C_\sigma$ be the set of the elementary circuits of $\GG_0^\sigma$ reachable from column node $n+1$. By~\eqref{ChiMinSigma}, we readily obtain
\begin{equation}
  v(\eps) = \chi_{n+1}(g_\eps) = \min_{\sigma \in \Sigma}
  \max_{(j^{\phantom{0}}_0, \dots, j_{l-1}, j_l) \in C_\sigma} \frac{1}{l}
  \Bigl( \sum_{0 \leq s \leq l-1} \!\!\!-\tilde{M}_{\sigma(j_s)j_s} +
  \tilde{N}_{\sigma(j_s) j_{s+1}}(\eps) \Bigr) \label{eq:chi}
\end{equation}
for $\eps \geq 0$. We deduce that:
\begin{lemma}\label{lemma:chi_piecewise_affine}\label{lemma:chi_non_increasing}
The function $\eps \mapsto v(\eps)$ is non-increasing and piecewise affine. 
\end{lemma}

\begin{proof}
The fact that $\eps \mapsto v(\eps)$ is piecewise affine straightforwardly follows from~\eqref{eq:chi}.

We next prove by
induction on $h$ that the function $\eps \mapsto (g_\eps^h(\vect{x}))_j$ is
non-increasing for any $\vect{x} \in \Real^n$ and $j\in [n]$. 
In the first place, observe that $\eps$ appears only negatively in the coefficients
$\tilde{N}_{ik}(\eps)$ of~\eqref{eq:operator}, so the function $\eps
\mapsto (g_\eps(\vect{x}))_j$ is non-increasing. Now consider $h \in \Nat$
and $\eps \leq \eps'$. By induction, we know that $g_{\eps'}^h(\vect{x}) \leq g_{\eps}^h(\vect{x})$ for any $\vect{x} \in
\Real^n$. 
Since the function $\vect{x} \mapsto
g_\eps(\vect{x})$ is order preserving, we have
$g_\eps(g_{\eps'}^h(\vect{x})) \leq g_\eps^{h+1}(\vect{x})$. Besides,
$g_\eps(g_{\eps'}^h(\vect{x})) \geq g_{\eps'}(g_{\eps'}^h(\vect{x}))) =
g_{\eps'}^{h+1}(\vect{x})$. It follows that $(g_{\eps'}^{h+1}(\vect{x}))_j \leq (g_\eps^{h+1}(\vect{x}))_j$ for all $j\in \oneto{n}$. 

We conclude that $\eps \mapsto v(\eps)$ is non-increasing as the limit of non-increasing functions. 
\end{proof}

If we assume that the numerical parts of the non-zero entries of
$M$, $N$, $\vect{p}$ and $\vect{q}$ are integers (this assumption is
obviously satisfied in the application to timed automata in
Section~\ref{sec:tropical_reachability_analysis}), the criterion of
Proposition~\ref{prop:emptiness_chi} can be determined by
evaluating 
$v(\eps)$ at $\eps = 0$ and at a small positive value.

\begin{proposition}\label{prop:emptiness_pre_criterion}
The tropical polyhedron with mixed constraints $\mcR$ is non-empty if, and only if, 
$v(0) \geq 1/(n+1)$ or $v(1/(n+1)^2) \geq 0$.
\end{proposition}

\begin{proof}
Each linear piece of the function $\eps \mapsto v(\eps)$ corresponds to an affine map given by the weight-to-length ratio of an elementary circuit of $\GG_\eps^\sigma$ for some strategy $\sigma$ for Player Min. In consequence, this affine map is of the form $\frac{\lambda - k \eps}{l}$, where:
\begin{enumerate}[(i)]
\item $l$ is the length of the circuit ($l \leq n+1$),
\item $\lambda$ is the sum of $2l$ integers given by the modulus, or their opposite, of some entries of the matrices $M$ and $N$ and of the vectors $\vect{p}$ and $\vect{q}$,
\item $k$ is the number of occurrences of $-\eps$ in the weight $N_{ij}(\eps)$ or $\vect{q}_i(\eps)$ of some arcs of the circuit (so $k \leq l$).
\end{enumerate}
Therefore, any non-differentiability point $\tilde{\eps}$ of the map $\eps \mapsto v(\eps)$ satisfies 
$(\lambda - k \tilde{\eps})/l = (\lambda' - k' \tilde{\eps})/l'$, 
where $l'$, $\lambda'$, and $k'$ have the same properties as $l$, $\lambda$, and $k$ above respectively. It follows that any positive non-differentiability point (if any) is of the form:
\[
\tilde{\eps} = \frac{\lambda' l - \lambda l'}{k' l - l' k} \enspace.
\]
Assume, without loss of generality, that $k/l < k'/l'$. Since $\tilde{\eps} > 0$, the numerator $\lambda' l - \lambda l'$ is a positive integer. This implies that $\tilde{\eps} \geq \eps^{*}$, where $\eps^{*} = 1/(n+1)^2$.

Consequently, the function $\eps \mapsto v(\eps)$ is affine on the
interval $[0, \epsilon^{*}]$. Since it is non-increasing by 
Lemma~\ref{lemma:chi_non_increasing}, we have $v(\eps) < 0$ for all
$\eps > 0$ if, and only if, $v(0) \leq 0$ (by continuity) and
$v({\eps^{*}}) < 0$. The application of Proposition~\ref{prop:emptiness_chi} 
shows that $\mcR$ is non-empty if, and only if, $v(0) > 0$ or $v({\eps^{*}})
\geq 0$. Since $v(0)$ is equal to the weight-to-length ratio of an
elementary circuit of $\GG_0^\sigma$, for some $\sigma \in \Sigma$, it
can be written as $\lambda'' /l''$, where $\lambda''$ and $l''$ have the same
properties as $\lambda$ and $l$ above. It follows that $v(0)$ is
positive if, and only if, it is greater than or equal to $1/(n+1)$. This
completes the proof. 
\end{proof}

Let $\GG'_0$ be the digraph obtained from $\GG_0$ by subtracting
$ 1/(n+1)$ from the weight of each arc connecting a row node with a column node. 
Then,
$v(0) -  (1/(n+1))$ is the value of the mean payoff game associated with
$\GG'_0$ when it starts from column node $n+1$. It follows that the
condition of Proposition~\ref{prop:emptiness_pre_criterion} holds if, and only
if, column node $n+1$ is winning (for Player Max) in one of the two games associated with 
$\GG'_0$ and $\GG_{1/(n+1)^2}$. Let $\GG^\star$ be the digraph obtained as
the disjoint union of $\GG'_0$ and $\GG_{1/(n+1)^2}$,
adding a special row node $0$ and two arcs, with zero weight, connecting
it with column nodes $n+1$ of $\GG'_0$ and $\GG_{1/(n+1)^2}$. The criterion of
Proposition~\ref{prop:emptiness_pre_criterion} can be restated as follows:

\begin{proposition}\label{prop:emptiness_criterion}
The tropical polyhedron with mixed constraints $\mcR$ is non-empty if, and only if, row node $0$ is a winning initial node (for Player Max) in the mean payoff game associated with $\GG^\star$.
\end{proposition}

As an immediate consequence, we obtain the following complexity result,
in which the equivalence (ii) $\Leftrightarrow$ (iii) extends
Theorem~3.2 of~\cite{AkianGaubertGutermanIJAC2011}, whereas the
equivalence (i) $\Leftrightarrow$ (iii) extends Theorem~18
of~\cite{AllamigeonGaubertKatzLAA2011} (only non-strict constraints are
considered there). 

\begin{theorem}\label{th:equivalence_MPG}
Under Assumption~\ref{ass:non_infty}, the following problems are (Karp) polynomial-time equivalent:
\begin{compactenum}[(i)]
\item\label{item:redundancy} 
deciding whether a mixed tropical affine inequality is implied by a system of such inequalities;
\item\label{item:emptiness} deciding whether a tropical polyhedron with mixed constraints is
  empty;
\item\label{item:mpg} determining whether a given initial node in a mean payoff game is winning.
\end{compactenum}
\end{theorem}

Problem~\eqref{item:mpg} is known to be in $\mathsf{NP} \cap
\mathsf{coNP}$, see~\cite{ZwickPatersonTCS1996}. We deduce from Theorem~\ref{th:equivalence_MPG} that Problems~\eqref{item:redundancy} and~\eqref{item:emptiness} both belong to the same complexity class ($\mathsf{NP}$ and $\mathsf{coNP}$ are closed under Karp reductions).
Whether Problem~\eqref{item:mpg} can be solved in polynomial time has
been an open question since the first combinatorial
algorithm~\cite{GurvichCMMP1988}.  Value iteration leads to a
pseudo-polynomial algorithm~\cite{ZwickPatersonTCS1996}. Several
algorithms rely on the idea of
strategy iteration~\cite{Howard1960}, applying various strategy improvement
rules, see in
particular~\cite{BjorklundVorobyovDAM2007,CochetGaubertGunawardena99,DhingraGaubertVALUETOOLS2006,JurdzinskiSODA2006}. A
remarkable example has recently been
constructed~\cite{Friedmann-AnExponentialLowerB} in which some common
rules lead to an exponential number of iterations. However, many
algorithms, including the one
of~\cite{CochetGaubertGunawardena99,DhingraGaubertVALUETOOLS2006}, are
known to have experimentally a small average case number of iterations
(growing sublinearly with the dimension), see the benchmarks
in~\cite{ChaloupkaESA2009}.

The \emph{support} of a tropically convex set $\CC \subset \maxplus^n$ is defined as the set $\supp(\CC)$ of indices $j \in \oneto{n}$ such that there exists $\vect{x} \in \CC$ verifying $\vect{x}_j \neq -\infty$. Note that $\supp(\CC)$ is the greatest subset $J \subset \oneto{n}$ such that $J = \{ j \in \oneto{n} \mid \vect{x}_j \neq -\infty\}$ for a certain $\vect{x} \in \CC$. 

Mean payoff games can be used to compute the support of the tropical polyhedron with mixed constraints $\mcR$. Indeed, $j$ belongs to the support of $\mcR$ if, and only if, $j \in \supp(\mcR_\eps)$ for some $\eps > 0$. By~\cite[Theorem~3.2]{AkianGaubertGutermanIJAC2011}, the fact that $j \in \supp(\mcR_\eps)$ is equivalent to $\chi_j(g_\eps)\geq 0$ and $\chi_{n+1}(g_\eps)\geq 0$, \ie\ to the fact that column nodes $j$ and $n+1$ are both winning initial nodes for Player Max in the game associated with $\GG_\eps$. Using the same arguments as above, it can be shown that there exists $\eps > 0$ such that $\chi_j(g_\eps) \geq 0$ if, and only if, $\chi_j(g_0) \geq 1/(n+1)$ or $\chi_j(g_{1/(n+1)^2}) \geq 0$. 

In consequence, the support of $\mcR$ can be computed by determining the winning initial nodes in the games associated with $\GG'_0$ and $\GG_{1/(n+1)^2}$. We point out that some policy iteration based algorithms, such as the one of~\cite{CochetGaubertGunawardena99,DhingraGaubertVALUETOOLS2006}, directly provide the cycle-time vector $\chi(g)$ of the dynamic programming operator $g$ of a mean payoff game (and so all the  winning initial nodes). 

\begin{remark}\label{remark:certificate}
Positional strategies for Player Max are defined symmetrically to the ones for Player Min, \ie \ as functions $\tau$ from row nodes to column nodes, such that for each row node $i$ there is an arc in $\GG_\eps$ connecting it with column node $\tau(i)$. Such a strategy $\tau$ induces a one-player game (now played by Player Min) whose associated digraph $\GG^\tau_\eps$ is obtained by removing from $\GG_\eps$ the arcs connecting row nodes $i$ with columns nodes $j$ such that $j \neq \tau(i)$. 

Positional strategies can be used as certificates to ensure that a mean payoff game is winning for one of the players, and these certificates can be checked in polynomial time. For instance, given a column node $j \in \oneto{n+1}$, a strategy $\sigma$ for Player Min satisfying $\chi_j(g_\eps^\sigma) \leq 0$ 
ensures that $\chi_j(g_\eps) \leq 0$ by~\eqref{ChiMinSigma}. 
In other words, column node $j$ is a winning initial node for Player Min in the game associated with $\GG_\eps$. 
Since as explained above $\chi_j(g_\eps^\sigma)$ is given by the maximal weight-to-length ratio of the circuits reachable from column node $j$ in $\GG_\eps^\sigma$, it can be checked that $\chi_j(g_\eps^\sigma)$ is less than or equal to $0$ in polynomial time using Karp's algorithm.
\end{remark}

\subsubsection{Polyhedra defined by systems with $+\infty$ coefficients.} 

We now deal with the case in which Assumption~\ref{ass:non_infty} is not satisfied. 
Suppose that the tropical polyhedron with mixed constraints $\mcR$ is defined by~\eqref{eq:r_system}, where now $N \in \smaxplus^{r \times n}$. Observe that we can still assume that $\vect{q} \in (\smaxplus \setminus \{+\infty\})^r$, because any inequality in which $\vect{q}_i$ is equal to $+\infty$ is trivial.

Given $I \subset \oneto{r}$, we denote by $\mcR_I$ the polyhedron defined by the inequalities
\[
M_{i1} \vect{x}_1 \mpplus \cdots  \mpplus M_{in} \vect{x}_n \mpplus \vect{p}_i \sleq
(\mpplus_{N_{ij} \neq +\infty} N_{ij} \vect{x}_j ) \mpplus \vect{q}_i \qquad \text{for}\ i \in I.
\]
The algorithm in Figure~\ref{fig:emptyness_test_general_case} determines whether $\mcR$ is empty by evaluating the emptiness of polyhedra of the form $\mcR_I$. To prove the correctness of this algorithm, we shall use the following lemma.

\begin{lemma}\label{lemma:invariant}
At each iteration of the loop at Line~\ref{line:loop}, we have $\mcR \subset \mcR_{I \cup I'}$.
\end{lemma}

\begin{proof}
We prove the lemma by induction on the number of iterations of the loop. Before the first iteration, we have $J = \oneto{n}$ and so $I \cup I' = \{ i \in \oneto{r} \mid \ N_{ij} \neq +\infty  \ \text{for all} \ j\in \oneto{n} \}$. Thus, the polyhedron $\mcR_{I \cup I'}$ is defined by a subsystem of $M \vect{x} \mpplus \vect{p} \sleq N \vect{x} \mpplus \vect{q}$, and the inclusion $\mcR \subset \mcR_{I \cup I'}$ is immediate. 

Now suppose that at iteration $k$ we have $\mcR \subset \mcR_{I_k \cup I'_k}$, and let $\vect{x} \in \mcR$. If the loop is iterated again, then the sets $J$, $I$ and $I'$ are respectively given by 
\begin{align*}
J_{k+1} & = \supp(\mcR_{I_k \cup I'_k}) \\
I_{k+1} & = I_k \cup I'_k \\
I'_{k+1} & = \{ i \not \in I_{k+1} \mid \ N_{ij} \neq +\infty \ \text{for all} \ j \in J_{k+1} \}
\end{align*}
In consequence, we have $\supp(\mcR) \subset J_{k+1}$, and so $\vect{x}_j = -\infty$ for $j \in \oneto{n} \setminus J_{k+1}$. 
Then, we deduce that $\vect{x}$ satisfies the inequality
\[
M_{i1} \vect{x}_1 \mpplus \cdots  \mpplus M_{in} \vect{x}_n \mpplus \vect{p}_i \sleq
( \mpplus_{j \in J_{k+1}} N_{ij} \vect{x}_j ) \mpplus \vect{q}_i
\sleq 
(\mpplus_{N_{ij} \neq +\infty} N_{ij} \vect{x}_j ) \mpplus \vect{q}_i
\]
for any $i\in I'_{k+1}$. Thus, we have $\vect{x} \in \mcR_{I'_{k+1}}$. Since $\vect{x} \in \mcR \subset \mcR_{I_k \cup I'_k} = \mcR_{I_{k+1}}$, we readily obtain $\vect{x} \in \mcR_{I_{k+1} \cup I'_{k+1}}$, which completes the proof.
\end{proof}

\begin{figure}
\begin{algorithmic}[1]
\REQUIRE polyhedron with mixed constraints $\mcR$ defined by the system $M \vect{x} \mpplus \vect{p} \sleq N \vect{x} \mpplus \vect{q}$
\ENSURE \TRUE{} if $\mcR$ is empty, \FALSE{} otherwise
\STATE $J := \oneto{n}$ 
\STATE $I := \emptyset$
\STATE $I' := \{ i \in \oneto{r} \mid \ N_{ij} \neq +\infty \ \text{for all} \ j \in J\}$
\WHILE{$I' \neq \emptyset$}\label{line:loop}
  \STATE $I := I \cup I'$\label{line:j_increase}
  \IF{$\mcR_I$ is empty}
    \RETURN \TRUE
  \ELSE
    \STATE $J := \supp(\mcR_I)$\label{line:i_update}
    \STATE $I' := \{ i \not \in I \mid \ N_{ij} \neq +\infty \ \text{for all} \ j \in J\}$\label{line:jprime_update}
  \ENDIF
\ENDWHILE
\RETURN \FALSE\label{line:false}
\end{algorithmic}
\caption{Determining whether a polyhedron with mixed constraints is empty in the general case.}\label{fig:emptyness_test_general_case}
\end{figure}

\begin{proposition}\label{prop:correctness}
The algorithm of Figure~\ref{fig:emptyness_test_general_case} is correct, and the number of iterations of the loop at Line~\ref{line:loop} is bounded by $\min(n,r)$.
\end{proposition}

\begin{proof}
Suppose that the algorithm returns \textbf{true}. Then $\mcR_I = \emptyset$, and by Lemma~\ref{lemma:invariant} we have $\mcR = \emptyset$. 

Now assume that \textbf{false} is returned. Let $\vect{x} \in \mcR_I$ be such that $\{j \in \oneto{n} \mid \vect{x}_j \neq -\infty\}=\supp (\mcR_I)=J$. We claim that $\vect{x} \in \mcR$. Indeed, since $I' = \emptyset$ (which is the condition to reach Line~\ref{line:false} and return \textbf{false}), we know that for all $i \not \in I$ there exists $j \in J$ such that $N_{ij} = +\infty$. As $\vect{x}_j \neq -\infty$ for $j\in J$, this means that the $i$-th inequality of the system $M \vect{x} \mpplus \vect{p} \sleq N \vect{x} \mpplus \vect{q}$ is satisfied (the right-hand side is equal to $+\infty$). As $\vect{x}$ also satisfies the inequalities of the system indexed by $i \in I$, the claim is proved.

Finally, observe that at each iteration of the loop, the set $I$ is strictly increased at Line~\ref{line:j_increase}. Similarly, if at Line~\ref{line:jprime_update} the set $I'$ is not empty, then necessarily the set $J$ has been strictly decreased at Line~\ref{line:i_update}. We deduce that the number of iterations is indeed bounded by $\min(n,r)$. 
\end{proof}

The idea behind the algorithm of Figure~\ref{fig:emptyness_test_general_case} can be used to build certificates. 
\begin{proposition}
The problem of determining whether a polyhedron with mixed constraints is empty belongs to $\mathsf{NP} \cap \mathsf{coNP}$.
\end{proposition}

\begin{proof}
A certificate that $\mcR \neq \emptyset$ can be provided by two sets $J \subset \oneto{n}$ and $I \subset \oneto{r}$ such that $\{ i \not \in I \mid \ N_{ij} \neq +\infty \ \text{for all} \ j \in J \} = \emptyset$, together with positional strategies ensuring that $\mcR_I \neq \emptyset$ and $\supp(\mcR_I) = J$ (see Remark~\ref{remark:certificate} and the discussion on supports which precedes it). The first property of the sets $I$ and $J$ can be checked in polynomial time, the same as the properties $\mcR_I \neq \emptyset$ and $\supp(\mcR_I) = J$ thanks to the positional strategies for the players. As shown in the proof of Proposition~\ref{prop:correctness}, this ensures that $\mcR$ is not empty. In consequence, the problem is in $\mathsf{coNP}$.

To certify that $\mcR = \emptyset$, we use a decreasing sequence $J_1 = \oneto{n}, J_2, \dots, J_k$ of subsets of $\oneto{n}$ and an increasing sequence $I_1, \dots, I_k$ of subsets of $\oneto{r}$ such that 
\[
I_l = \{ i \in \oneto{r} \mid  N_{ij} \neq +\infty \ \text{for all} \ j \in J_l \} \; ,
\]
for all $l \in \oneto{k}$, together with positional strategies ensuring that $\mcR_{I_l} \neq \emptyset$ and $J_{l+1} = \supp(\mcR_{I_l})$ for $l\in \oneto{k-1}$, and that $\mcR_{I_k} = \emptyset$. Since $I_{l+1} = I_l \cup \{ i \not \in I_l \mid \ N_{ij} \neq +\infty \ \text{for all} \ j \in J_{l+1} \}$ for $l\in \oneto{k-1}$, it can be shown by induction on $l \in \oneto{k}$ that $\mcR \subset \mcR_{I_l}$, using the same technique as in the proof of Lemma~\ref{lemma:invariant}. Thus, these certificates allow to prove that $\mcR = \emptyset$, and they can be checked in polynomial time. This completes the proof.
\end{proof}

\subsection{Polynomial-time weak redundancy elimination}\label{subsec:weak_criterion}

Since no polynomial-time algorithm is known to evaluate the 
criteria given in Section~\ref{Subsection:EquivMPG}, we also develop a sufficient criterion for which a potentially faster algorithm exists. It consists in checking whether $\vect{e} \vect{x} \mpplus g \sleq \vect{f} \vect{x} \mpplus h$ is a linear combination of the inequalities in $A \vect{x} \mpplus \vect{c} \sleq B \vect{x} \mpplus \vect{d}$. Note that in general, this condition is not necessary for~\eqref{eq:implication} to hold (the tropical analogue of Farkas' lemma given in~\cite{AllamigeonGaubertKatzLAA2011} shows that taking tropical linear combinations does not suffice to deduce all valid inequalities).

Now we see the constraint
$\vect{e} \vect{x} \mpplus g \sleq \vect{f} \vect{x} \mpplus h$ as the
$(2n+2)$-dimensional row vector $\vect{v} \mydef (\vect{e},g,\vect{f},h)$. Similarly, we
introduce the matrix $R \in \smaxplus^{p \times (2n+2)}$, whose rows
are given by the vectors $(A_i,\vect{c}_i,B_i,\vect{d}_i)$, 
for $i \in \oneto{p}$, where $A_i$ and $B_i$ denote the $i$-th rows of 
$A$ and $B$ respectively. We are reduced to the problem of determining
whether there exists a $p$-dimensional row vector $\vect{w}$ (with entries in $\smaxplus$) such that
$\vect{v} = \vect{w} R$. Without loss of generality, we
assume that no row of $R$ is identically zero. We propose 
a method based on residuation theory (see
\eg~\cite{GaubertPlusSTACS1997}). 
Given $x \in \smaxplus$, the self-map $z \mapsto z x$ on $\smaxplus$ can be shown to be residuated, meaning that
for each $y \in \smaxplus$ there exists a maximal element of the set $\left\{z \in \smaxplus \mid z x \sleq y\right\}$, denoted by $y/x$. Indeed,
the later element is given by: 
\[
y \mapsto y / x := 
\begin{cases}
\myul{(\abs{y} - \abs{x})} & \text{if} \ x \in \R, \ y \in \pert \; , \\
\abs{y} - \abs{x} & \text{otherwise},
\end{cases}
\]
where it is used the conventions $\alpha - (\mpzero) = +\infty$ for $\alpha \in \cmaxplus$, 
$\alpha - (+\infty) = \mpzero$ for $\alpha \in \maxplus$, and $(+\infty) - (+\infty) = +\infty$.

\begin{proposition}\label{prop:weak_implication}
Define the $p$-dimensional row vector $\vect{w}^{*}$ by
\[
\vect{w}^*_i\mydef \min_{j \in \oneto{2n+2}} (\vect{v}_j/R_{ij}) 
\]
for $i \in \oneto{p}$. Then, the inequality $\vect{e} \vect{x} \mpplus g \sleq \vect{f} \vect{x}
\mpplus h$ is a linear combination of the inequalities in the system $A
\vect{x} \mpplus \vect{c} \sleq B \vect{x} \mpplus \vect{d}$ if, and
only if, $\vect{v} = \vect{w}^{*} R$. 
\end{proposition}

The principle of Proposition~\ref{prop:weak_implication} is that $\vect{w}^{*}$ can be shown to be the greatest solution of $\vect{w} R \sleq \vect{v}$. Thus, there exists a solution to $\vect{v} = \vect{w} R$ if, and only if, the equality is satisfied for $\vect{w} = \vect{w}^*$. It follows that the criterion of Proposition~\ref{prop:weak_implication} can be checked efficiently, in time $O(n \times p)$.

\subsection{Complexity of successive Fourier-Motzkin eliminations}\label{subsec:complexity}

As discussed in the beginning of Section~\ref{subsec:elimination_step}, propagating redundant inequalities may produce $O(p^{2^k})$ constraints after $k$ calls to Fourier-Motzkin elimination method (recall that $p$ refers to the number of constraints defining the initial polyhedron). We claim that, in the case of closed tropical polyhedra, the number of constraints remains simply exponential at every Fourier-Motzkin elimination when the weak redundancy criterion is used. To see this, let $\PP_0 = \PP$, $\PP_1$, \dots, $\PP_k$ be the closed tropical polyhedra arising during a sequence of $k$ calls to Fourier-Motzkin elimination followed by the weak inequality redundancy elimination of Proposition~\ref{prop:weak_implication}.

\begin{proposition}\label{prop:simple_exponential}
There exists a constant $K$ (independent from $k$, $n$, and $p$) such that for all $l \in [k]$, the number of inequalities describing $\PP_l$ obtained using the weak inequality redundancy elimination is bounded by 
\begin{equation}
K (n-l+1) p^{\floor{n/2} \floor{(n-l)/2}} \enspace . \label{eq:bound}
\end{equation}
\end{proposition}

\begin{proof}
By the tropical analogue of McMullen's upper bound theorem~\cite{AllamigeonGaubertKatzJCTA2011}, we know that the number $q$ of extreme generators (points and rays) of $\PP$ is bounded by $U(p+n+1, n)$, where $U(p,n)$ is the number of facets of (classical) cyclic polytopes with $p$ extreme points in dimension $n$. In particular, $q$ is bounded by $p^{\floor{n/2}}$ when $n$ and $p$ are sufficiently large (see Appendix~\ref{sec:bound}).

It is straightforward to see that the extreme generators of every polyhedron $\PP_l$ ($l \in [k]$) arise as the projection of some of the extreme generators of $\PP$. In consequence, the number of extreme generators of each $\PP_l$ is bounded by $q$. 

Dually, the set of inequalities satisfied by all the points of a tropical polyhedron $\QQ$ forms a tropical polyhedral cone called the \emph{polar cone} of $\QQ$, see~\cite{AllamigeonGaubertKatzLAA2011}. As proved in~\cite{AllamigeonGaubertKatzJCTA2011}, if $\QQ$ is generated by $q$ points and rays in dimension $n-l$, the number of (non-trivial) extreme rays of its polar is bounded by $(n-l+1) (U(q+n-l+1,n-l)-n+l+2)$. In particular, this also bounds the size of any description of $\QQ$ by linearly independent inequalities. It can be shown that there exists a constant $K' > 0$ such that the latter quantity is bounded by $K' (n-l+1) q^{\floor{(n-l)/2}}$ for all values of $q$ and $n-l$ (see Appendix~\ref{sec:bound}). It follows that the representation by inequalities obtained after the application of the weak redundancy criterion is bounded by a quantity of the form~\eqref{eq:bound}.
\end{proof}

As a consequence, after the $(l+1)$-th call to Fourier-Moztkin elimination, the number of inequalities (defining $\PP_{l+1}$) is bounded by the square of~\eqref{eq:bound}, and the weak redundancy criterion can be applied in time $O((n-l)^3 p^{n (n-l)/2})$. 
It follows that the complexity of $k$ successive calls to tropical
Fourier-Motzkin elimination using the weak redundancy criterion can be
bounded by 
\[
O(k n^3 p^{n^2/2}) \enspace .
\]
Therefore, the time complexity is only exponential in the worst case. 

Since the criterion provided
by Theorem~\ref{th:equivalence_MPG} may eliminate further
inequalities, the number of inequalities describing $\PP_l$ obtained using this criterion instead of the weak one is also bounded by~\eqref{eq:bound}. However, this has to be balanced with the potentially
greater cost of the associated algorithm for solving mean payoff games.

\section{Tropical forward exploration for timed automata}
\label{sec:tropical_reachability_analysis} 

Timed automata is one of the formalisms used for modelling and
verification of real-time systems. As an application of the methods
developed in this paper, we show how the forward exploration algorithm
for timed automata~\cite{journal/tcs/AlurDd94} can be implemented using
tropical polyhedra with mixed constraints as symbolic states.  This
algorithm is used to solve the reachability problem, to which most
verification problems for timed automata can be reduced,
see~\cite{DBLP:journals/iandc/HenzingerNSY94,DBLP:journals/iandc/AlurCD93}.

We first recall some notions concerning timed automata, and illustrate
with an example the drawbacks (mentioned in the introduction) of using
zones or closed tropical polyhedra as symbolic states.  Note that zones
are also used as symbolic states in the verification of other real-time
models such as \eg~timed Petri
nets~\cite{conf/ajwsm/Bowden96,DBLP:conf/tacas/LimeRST09}, hence 
the corresponding algorithms can potentially benefit from our results too.

We consider a timed automaton over the set of
clocks $C=\{\vect{x}_1,\dots, \vect{x}_n\}$. It is represented by a directed graph whose nodes correspond to locations ($l_0, l_1, \dots$). We denote by $l_0$ the initial location and by $L$ the set of locations. 
An edge from location $l$ to $l'$ is denoted by $l \tto[
\phi, r] l'$, where $\phi$ represents a clock constraint and $r$ a set of reset
operations. More precisely, $\phi$ is a (possibly empty) conjunction of
atomic clock constraints of the form $\vect{x}_i \bowtie k$ and
$\vect{x}_i\bowtie k+ \vect{x}_j$, where $\vect{x}_i, \vect{x}_j\in C$,
$k \in \Int$, and $\mathord{\bowtie} \in \{\mathord{<}, \mathord{\leq},
\mathord{=},\mathord{\geq}, \mathord{>}\}$. Besides, $r$ is defined as a
partial function from $C$ to $\Nat$, meaning that $\vect{x}_i$ is mapped
to $k$ when the clock $\vect{x}_i$ is reset to the value $k$. The set of edges is denoted by $E$. 
Each location $l$ can be additionally labeled by a clock constraint
$\theta(l)$. Then, the states of the automaton are of the form $(l, v)$,
where $l \in L$ and $v : C \to \Realnn$ is such that $(v(\vect{x}_1),
\dots, v(\vect{x}_n))$ satisfies the constraint $\theta(l)$.  The evolution
of the system (\ie\ the \emph{semantics} of the automaton) is expressed
as a transition relation on these states, denoted by $\leadsto$. Transitions can be of two kinds:
\begin{itemize}[\textbullet]
\item[\emph{Delays}:]  where clock values increase synchronously at a
  given location. More precisely, $(l,v) \leadsto (l, v')$ if there exists $t \geq 0$ such that 
  $v'(\vect{x}_i) := v(\vect{x}_i) + t$ for all $i \in [n]$,
  and $(v(\vect{x}_1)+ t', \dots, v(\vect{x}_n)+ t')$ satisfies
  $\theta(l)$ for all $t' \in [0, t]$.
\item[\emph{Switches}:]  which are governed by the edges $l \tto[ \phi, r]
  l'$. In this case, we have $(l,v) \leadsto (l', v')$ if $(v(\vect{x}_1),
  \dots, v(\vect{x}_n))$ satisfies the constraints $\phi$, and
  $v'(\vect{x}_i) := r(\vect{x}_i)$ if $r$ is defined on $\vect{x}_i$,
  $v'(\vect{x}_i) := v(\vect{x}_i)$ otherwise.
\end{itemize}

The basic problem in the verification of timed automata is (untimed) reachability:
is a final location $l_f$ of the automaton reachable? More precisely, 
the reachability problem consists in determining whether there exist clock values $v_f: C\to \Realnn$ and a
finite path $(l_0,v_0)\leadsto^*( l_f, v_f)$ of transitions in the
automata, where $v_0$ is the function which maps every $\vect{x}_i$ to $0$. 

\begin{figure}[tp]
\centering
\begin{tikzpicture}[>=stealth',initial text=,every state/.style={minimum size=0.4cm},convex/.style={draw=lightgray,fill=lightgray,fill opacity=0.7},convexborder/.style={thick},scale=.9]
\coordinate (cl0) at (2,1.75);
\coordinate (cl1) at (0.5,0);
\coordinate (cl2) at (3.5,0);
\coordinate (cl3) at (2,-1.75);
\coordinate (cl4) at (2,-3.5);

\node[state,initial] (l0) at (cl0) {$l_0$};
\node[state] (l1) at (cl1) {$l_1$};
\node[state] (l2) at (cl2) {$l_2$};
\node[state] (l3) at (cl3) {$l_3$};
\node[state] (l4) at (cl4) {$l_f$};
\path[->] (l0) edge node [above left=-0.1cm] {$\vect{x}_1:= 0$} (l1);
\path[->] (l0) edge node [above right=-0.1cm] {$\vect{x}_2:= 0$} (l2);
\path[->] (l1) edge node [below left=-0.05cm] {$\vect{x}_2> 1$} (l3);
\path[->] (l2) edge node [below right=-0.05cm] {$\vect{x}_1> 1$} (l3);
\path[->] (l3) edge node [right] {$\vect{x}_1 \leq 1 \wedge \vect{x}_2 \leq 1$} (l4);

\begin{scope}[shift={($(cl0)+(7,0)$)}]
\node at (0,0) {$l_0:$};
\begin{scope}[scale=0.35,shift={(2,-1.2)}]
\draw[gray!40,very thin] (-0.5,-0.5) grid (2.5,2.5);
\draw[convexborder] (0,0) -- (2.5,2.5);
\end{scope}
\end{scope}

\begin{scope}[shift={($(cl1)+(7,0)$)}]
\node at (0,0) {$l_1:$};
\begin{scope}[scale=0.35,shift={(2,-1.2)}]
\draw[convexborder] (0,0) -- (2.5,2.5);
\draw[gray!40,very thin] (-0.5,-0.5) grid (2.5,2.5);
\filldraw[convex] (0,0) -- (0,2.5) -- (2.5,2.5) -- cycle;
\draw[convexborder] (2.5,2.5) -- (0,0) -- (0,2.5);
\end{scope}
\end{scope}

\begin{scope}[shift={($(cl2)+(7,0)$)}]
\node at (0,0) {$l_2:$};
\begin{scope}[scale=0.35,shift={(2,-1.2)}]
\draw[gray!40,very thin] (-0.5,-0.5) grid (2.5,2.5);
\filldraw[convex] (0,0) -- (2.5,2.5) -- (2.5,0) -- cycle;
\draw[convexborder] (2.5,0) -- (0,0) -- (2.5,2.5);
\end{scope}
\end{scope}

\begin{scope}[shift={($(cl3)+(7,0)$)}]
\node at (0,0) {$l_3:$};
\begin{scope}[scale=0.35]
\begin{scope}[shift={(2,-1.2)}]
\draw[gray!40,very thin] (-0.5,-0.5) grid (2.5,2.5);
\filldraw[convex] (0,2.5) -- (0,1) -- (1,1) -- (1,0) -- (2.5,0) -- (2.5,2.5) -- cycle;
\draw[convexborder] (0,2.5) -- (0,1) (1,0) -- (2.5,0);
\end{scope}
\node at (6,-0.2) {$=$};
\begin{scope}[shift={(8,-1.2)}]
\draw[gray!40,very thin] (-0.5,-0.5) grid (2.5,2.5);
\filldraw[convex] (0,2.5) -- (0,1) -- (1,1) -- (2.5,2.5) -- cycle;
\draw[convexborder] (0,2.5) -- (0,1) (1,1) -- (2.5,2.5);
\end{scope}
\node at (12,-0.2) {$\cup$};
\begin{scope}[shift={(14,-1.2)}]
\draw[gray!40,very thin] (-0.5,-0.5) grid (2.5,2.5);
\filldraw[convex] (1,1) -- (1,0) -- (2.5,0) -- (2.5,2.5) -- cycle;
\draw[convexborder] (1,1) -- (2.5,2.5) (1,0) -- (2.5,0);
\end{scope}
\end{scope}
\end{scope}

\begin{scope}[shift={($(cl4)+(7,0)$)}]
\node at (0,0) {$l_f:$};
\begin{scope}[scale=0.35,shift={(2,-1)}]
\draw[gray!40,very thin] (-0.5,-0.5) grid (2.5,2.5);
\end{scope}
\end{scope}
\end{tikzpicture}
\caption{Left: Timed automaton. Right: Symbolic states accumulated during forward exploration, shown after delays (black borders are included in the depicted gray regions).\label{fi:ta-disj}\label{fig:ta-disj}}
\end{figure}
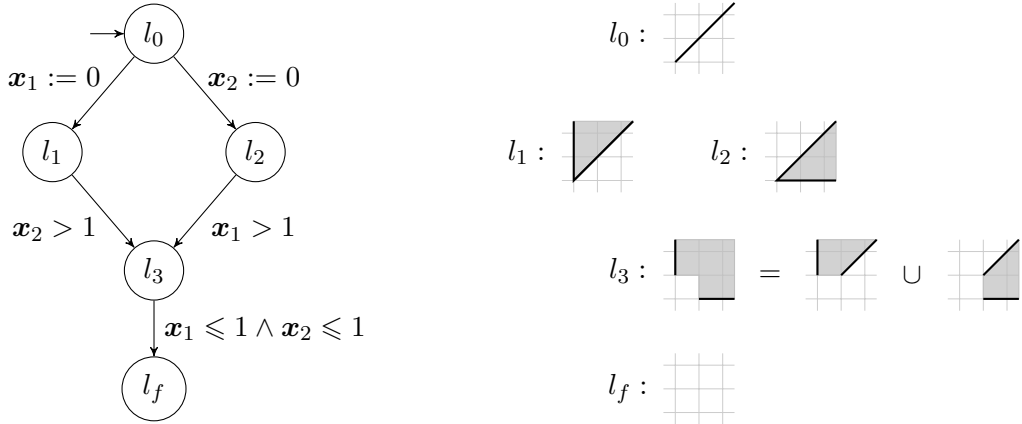

\begin{example}
Consider the timed-automaton fragment depicted in Figure~\ref{fi:ta-disj}, which involves two clocks $\vect{x}_1$ and $\vect{x}_2$.
Its edges are labeled by the reset operations (for instance, $\vect{x}_1 := 0$) and/or the constraints
on clocks (for instance, $\vect{x}_2 > 1$). As mentioned above, the initial location is $l_0$, and the two clocks are
initialized to $0$. 

The diagrams on the right-hand side of Figure~\ref{fi:ta-disj} depict the
symbolic states of this automaton, \ie\ the sets of states (clock values) which can
arise at each location. For instance, at location $l_0$ we recover the initial state $\vect{x}_1 =
\vect{x}_2 = 0$, and all the other states $\vect{x}_1 = \vect{x}_2 = t
> 0$ arise as time goes by while staying at the same
location. Similarly, the states at location $l_1$ are obtained from the
ones at location $l_0$ by resetting clock $\vect{x}_1$, \ie\ setting
$\vect{x}_1 = 0$, while $\vect{x}_2 \geq 0$ is not affected. Then, as
time goes by, we get all the states satisfying $0 \leq \vect{x}_1 
\leq \vect{x}_2$. Note that the final location $l_f$ is not reachable in this example.

The symbolic states at locations $l_0$, $l_1$ and $l_2$ can be
represented exactly by zones, but the one at location $l_3$
cannot. Hence, the symbolic state at $l_3$ has to be
split, potentially doubling the number of symbolic states to be
visited after. If several such timed-automaton fragments are
concatenated, it is easy to see that this splitting of symbolic states
may lead to a situation where an exponential number of zones have to be
used to determine that the final location is not
reachable. Alternatively, the symbolic state at location $l_3$ could be over-approximated
by a single zone
\tikz[baseline=3pt,scale=0.18,convex/.style={draw=lightgray,fill=lightgray,fill
  opacity=0.7},convexborder/.style={thick}]{ \draw[gray!40,very thin]
  (-0.5,-0.5) grid (2.5,2.5); \filldraw[convex] (0,2.5) -- (0,0) --
  (2.5,0) -- (2.5,2.5) -- cycle; \draw[convexborder] (0,2.5) -- (0,0) --
  (2.5,0); }, but in this case we cannot certify anymore that the final location $l_f$ is not reachable. 
On the other hand, the use of closed tropical polyhedra to
over-approximate the union of the sets of states arising from $l_1$ and $l_2$
provides the closure  
\tikz[baseline=3pt,scale=0.18,convex/.style={draw=lightgray,fill=lightgray,fill opacity=0.7},convexborder/.style={thick}]{%
\draw[gray!40,very thin] (-0.5,-0.5) grid (2.5,2.5);
\filldraw[convex] (0,2.5) -- (0,1) -- (1,1) -- (1,0) -- (2.5,0) --
(2.5,2.5) -- cycle; \draw[convexborder] (0,2.5) -- (0,1) -- (1,1) --
(1,0) -- (2.5,0); } of this union, which contains the
point $(1,1)$. Then, the final location $l_f$ becomes reachable, while it
should not if strict constraints were correctly handled. 
\end{example}

The reachability problem can be solved using a symbolic forward
exploration algorithm, first given
in~\cite{DBLP:journals/iandc/HenzingerNSY94}, which is still used in
state-of-the-art tools.  This algorithm, shown in
Figure~\ref{al:freach}, is usually implemented using zones (or DBMs).
The algorithm explores sets of reachable states, representing them
as symbolic states using zones, and performing symbolic delay and switch
operations on them. We do not discuss this algorithm further, but we
point out that any class of symbolic states can be used, provided that
it supports the operations \texttt{is\_empty}, \texttt{is\_included},
\texttt{intersect}, \texttt{reset}, and \texttt{delay} used in the
above algorithm. We now detail
the definition of these operations, and show how to implement them over
tropical polyhedra with mixed constraints.

\begin{figure}[tbp]
  \begin{algorithmic}[1]
    \REQUIRE timed automaton $( L, l_0, C, \theta, E)$, $l_f\in L$
    \ENSURE \TRUE\ if $\exists v_f: C\to \Realnn:( l_0,
    v_0)\leadsto^*( l_f, v_f)$, \FALSE\ otherwise
    \STATE $\textit{Waiting}:= \{( l_0, \texttt{intersect}_{ \theta( l_0)}(
    \texttt{delay}( \{ v_0\})))\}$; $\textit{Passed}:= \emptyset$
    \WHILE{$\textit{Waiting} \neq \emptyset$}
    \STATE Choose and remove $(l,V)$ from \textit{Waiting}
    \IF{$ l = l_f$}
    \RETURN \TRUE
    \ENDIF
    \IF{(\NOT $\texttt{is\_included}(V, V')$) for all $(l,V') \in
      \textit{Passed}$}  \label{al:freach.include} 
    \STATE $\textit{Passed}:= \textit{Passed} \cup \{(l,V)\}$
    \label{al:freach.addpassed}
    \FORALL{$l\tto[ \phi, r] l'$}
    \STATE \label{al:freach.explore} $V':= \texttt{intersect}_{ \theta( l')}(
    \texttt{delay}( \texttt{reset}_r( \texttt{intersect}_\phi( V))$
    \IF{\NOT $\texttt{is\_empty}( V')$}
    \STATE $\textit{Waiting}:= \textit{Waiting}\cup\{( l', V')\}$
    \label{al:freach.addwaiting}
    \ENDIF
    \ENDFOR
    \ENDIF
    \ENDWHILE
    \RETURN \FALSE
  \end{algorithmic}
  \caption{%
    \label{al:freach}
    The symbolic forward reachability algorithm for timed automata.}
\end{figure}

The operation $\mathtt{is\_empty}(\PP)$ determines whether the
polyhedron $\PP$ is empty. It is implemented using
the methods described in Section~\ref{Subsection:EquivMPG}, \ie \ through a reduction to mean payoff
games. The operation $\mathtt{is\_included}(\PP_1,\PP_2)$ checks if
$\PP_1 \subset \PP_2$. This can be performed using the procedure for
deciding implications of Section~\ref{Subsection:EquivMPG}, by
determining whether the system of inequalities defining $\PP_1$ implies
each defining inequality of the polyhedron $\PP_2$. The operation
$\mathtt{intersect}_\psi(\PP)$ computes the intersection of $\PP$ with
the constraints in $\psi$. It is simply defined by appending the
inequalities in $\psi$ to the system defining $\PP$, where any strict
constraint in $\psi$ is encoded as an inequality over $\smaxplus$ using
elements of the form $\myul{\lambda}$.
The operation $\mathtt{reset}_{\vect{x}_i := k}(\PP)$ consists in computing the polyhedron 
\[
\{ \vect{y} \in \maxplus^n \mid \vect{x} \in \PP, \ \vect{y}_j = \vect{x}_j \ \text{if}\ j \neq i, \ \vect{y}_i = k\} \ .
\]
It can be obtained by eliminating $\vect{x}_i$ in the system of constraints defining $\PP$ using Fourier-Motzkin elimination, and then intersecting the resulting
polyhedron with the constraint $\vect{x}_i= k$ (encoded as two inequalities $\vect{x}_i \sleq k$ and $k \sleq \vect{x}_i$).
Finally, the operation $\mathtt{delay}(\PP)$ consists in converting the polyhedron $\PP$
into the set $\{ \lambda \vect{x} \mid \vect{x} \in \PP, \lambda \geq 0
\}$ (recall that $\lambda \vect{x}$ corresponds to the vector with
entries $\lambda + \vect{x}_i$).
Assuming that $\PP$ is given by the system $A \vect{x} \mpplus
\vect{c} \sleq B\vect{x}\mpplus \vect{d}$, we first let $\QQ$ be the
polyhedron defined by $A \vect{x} \mpplus \lambda \vect{c} \sleq B
\vect{x} \mpplus \lambda \vect{d}$ and $0 \sleq \lambda$, and then apply Fourier-Motzkin
elimination on $\lambda$ to get $\mathtt{delay}(\PP)$.  To prove this
algorithm is correct, observe that:
\begin{align*}
\mathsf{delay}(\PP) & = \{ \lambda \vect{x} \in \maxplus^n \mid 0 \sleq \lambda , \ A
\vect{x} \mpplus \vect{c} \sleq B \vect{x} \mpplus \vect{d} \} \\
& = \{ \vect{x} \in \maxplus^n \mid 0 \sleq \lambda , \ A (\lambda^{-1}\vect{x}) \mpplus \vect{c} \sleq B (\lambda^{-1}\vect{x}) \mpplus \vect{d} \} \\
& = \{ \vect{x} \in \maxplus^n \mid 0 \sleq \lambda , \ A \vect{x} \mpplus \lambda \vect{c} \sleq B\vect{x}\mpplus
\lambda \vect{d} \} 
\enspace.
\end{align*}

To combat state space explosion, symbolic states are also equipped with
an \emph{over-approximating union} operator. In this way, symbolic states which
are reached through different paths
may be recombined, leading to a significant reduction in the number of
symbolic states the forward exploration algorithm has to consider. 
Given two polyhedra with mixed constraints $\PP, \PP'\subset \maxplus^n$, 
the over-approximation union operator $\mathtt{over\_approx}(\PP,\PP')$ 
is defined as the tropical convex hull
of $\PP \cup \PP'$. Note that given systems of mixed inequalities describing $\PP$ and  $\PP'$, a system describing $\mathtt{over\_approx}(\PP,\PP')$ can be computed by means of Proposition~\ref{prop:over_approx}.
As this involves $2n+2$ calls to tropical Fourier-Motzkin elimination, it is crucial to implement some redundancy elimination. 

Observe that the error introduced using the operation $\mathtt{over\_approx}$ over polyhedra with mixed constraints is smaller than using zone-based over-approximation. Indeed, zones are tropically convex, and thus any zone containing $\PP \cup \PP'$ also contains the tropical convex hull of $\PP \cup \PP'$. In the example of Figure~\ref{fig:ta-disj}, the
over-approximating union by polyhedra with mixed constraints of the two sets of states arising at $l_3$ is exact, and given by the polyhedron defined by $\vect{x}_1, \vect{x}_2 \sgeq 0$ and $1 \sleq \myul{0} \vect{x}_1 \mpplus \myul{0}
\vect{x}_2$, or equivalently, $1 < \max(\vect{x}_1, \vect{x}_2)$.

We have implemented a prototype of the forward exploration algorithm
based on tropical polyhedra with mixed constraints. The algorithms of
Section~\ref{sec:fourier_motzkin} and the operations described above
have been implemented within the OCaml library~\tplib{}~\cite{TPLib},
whose purpose is to provide algorithms for tropical polyhedra.
It relies on the library~\mpglib{}~\cite{MPGLib}, which implements the
algorithm in~\cite{DhingraGaubertVALUETOOLS2006} for solving mean payoff
games by policy iteration.
Our prototype successfully checks that location $l_f$ is not
reachable in the timed automaton of Figure~\ref{fi:ta-disj}. In future works, we plan to apply our method on more representative examples taken from real case studies, and to compare it in terms of performance/precision with state-of-the-art tools such as \textsc{Uppaal}~\cite{DBLP:conf/qest/BehrmannDLHPYH06}.

\section*{Acknowledgments}

The authors warmly thank Alexandre David for helpful discussions on the algorithms involved in the verification of timed automata. The authors are also grateful to the anonymous referee for his comments and suggestions which helped to improve the paper.

\bibliographystyle{alpha}      
\newcommand{\etalchar}[1]{$^{#1}$}

\appendix
\section{Additional details for the proof of Proposition~\ref{prop:simple_exponential}}\label{sec:bound}

We first recall that:
\[
U(p,n) =
\begin{cases}
\binomial{p-\floor{n/2}}{\floor{n/2}}+ 
\binomial{p-\floor{n/2}-1}{\floor{n/2}-1} 
& \text{for} \ n \ \text{even,} \\
2\binomial{p-\floor{n/2}-1}{\floor{n/2}}  
& \text{for} \ n \ \text{odd.}
\end{cases}
\]
Thus, if $n$ is even,
\[
U(p+n+1,n) = \binom{p+1+\floor{n/2}}{\floor{n/2}} + \binom{p+\floor{n/2}}{\floor{n/2}-1} \; ,
\]
and if $n$ is odd,
\[
U(p+n+1,n) = 2\binom{p+1+\floor{n/2}}{\floor{n/2}} \ .
\]
In both cases, it can be easily checked that
\[
U(p+n+1,n) \leq 2 \binom{p+1+\floor{n/2}}{\floor{n/2}} \ .
\]
We claim that for $p$ and $n$ sufficiently large, 
\begin{equation}
U(p+n+1, n) \leq p^{\floor{n/2}} \ . \label{eq:bound1}
\end{equation}
Let $m := \floor{n/2}$. By Stirling approximation formulas, we know that for all positive integer $h$, 
\[
\sqrt{2\pi h} (h/e)^h \leq h! \leq e \sqrt{h} (h/e)^h 
\]
As a result,
\begin{align*}
\binom{p+1+m}{m} & \leq \frac{e}{2\pi} \sqrt{\frac{p+1+m}{(p+1) m}} \Bigl(1+\frac{p+1}{m}\Bigr)^m \Bigl(1+\frac{m}{p+1}\Bigr)^{p+1} \\
& \leq \frac{1}{2} \Bigl(1+\frac{p+1}{m}\Bigr)^m \Bigl(1+\frac{m}{p+1}\Bigr)^{p+1}
\end{align*}
when $p \geq 1$ and $m \geq 2$. We next show that $\bigl(1+\frac{p+1}{m}\bigr)^m \bigl(1+\frac{m}{p+1}\bigr)^{p+1}$ is bounded by $p^m$ by considering their logarithm:
\[
m \ln \Bigl(1+\frac{p+1}{m}\Bigr) + (p+1) \ln \Bigl(1+\frac{m}{p+1}\Bigr) \leq m \ln \Bigl(1+\frac{p+1}{m}\Bigr) + m = m \ln \Bigl(\frac{e(p+1+m)}{m} \Bigr) \leq m \ln p
\]
as soon as $e (p+m+1) \leq p m$. The latter condition is satisfied if $p \geq 6$ and $m \geq 6$. This shows that~\eqref{eq:bound1} holds when $p \geq 6$ and $n \geq 12$.

Using the same arguments, we obtain that:
\[
U(q+n-l+1, n-l) \leq q^{\floor{(n-l)/2}}
\]
as soon as $q \geq 6$ and $n-l \geq 12$. In any case, we can find a constant $K' > 0$ such that for all values of $q$ and $n-l$, 
\[
U(q+n-l+1, n-l) \leq K' q^{\floor{(n-l)/2}}
\]
Assuming that $q \leq U(p+n+1, n)$, it follows that for a certain constant $K'' > 0$,
\[
(n-l+1) (U(q+n-l+1,n-l)-n+l+2) \leq K'' (n-l+1) p^{\floor{n/2} \floor{(n-l)/2}}
\]
if $p \geq 6$ and $n \geq 12$. This allows to show that there exists $K > 0$ such that for all $p$, $n$, and $l \in [n-1]$,
\[
(n-l+1) (U(q+n-l+1,n-l)-n+l+2) \leq K (n-l+1) p^{\floor{n/2} \floor{(n-l)/2}} \; .
\]

\end{document}